\documentclass[12pt, reqno]{amsart}

\usepackage{amsfonts}
\usepackage{amssymb}
\usepackage{amsthm}
\usepackage{amsmath}
\usepackage{enumerate}
\usepackage{tikz-cd}
\usepackage{mathtools}
\usepackage[style = nature, sorting = nyt]{biblatex}
\usepackage{graphicx}

\newtheorem{thm}{Theorem}[section]

\newtheorem{lemma}[thm]{Lemma}

\newtheorem{rmk}[thm]{Remark}

\newtheorem{prop}[thm]{Proposition}
\newtheorem{conj}[thm]{Conjecture}
\newtheorem{claim}[thm]{Claim}

\newtheorem{ack}[thm]{Acknowledgements}
\numberwithin{equation}{section}

\DeclareMathOperator{\F}{\mathbf{F}}

\begin{document}
\title{The Smale conjecture and min-max theory}
\author{Daniel Ketover}\address{Rutgers University\\  Busch Campus - Hill Center \\ 110 Freylinghausen Road, Piscataway NJ 08854 USA}
\author{Yevgeny Liokumovich}\address{University of Toronto Mississauga\\ Mathematical and Computational Sciences  \\ 3359 Mississauga Road \\ Mississauga, ON, L5L 1C6}
\thanks{D.K. was partially supported by NSF-PRF DMS-1401996. Y.L. was partially supported by
NSERC Discovery Grant and Sloan Fellowship.}

\maketitle

\begin{abstract}
We give a new proof of the Smale conjecture for $\mathbb{RP}^3$ and all lens spaces using minimal surfaces and min-max theory.    For $\mathbb{RP}^3$,  the conjecture was first proved in 2019 by Bamler-Kleiner using Ricci flow.  
\end{abstract}

\section{Introduction}

In 1959,  S. Smale proved that for the round $2$-sphere $\mathbb{S}^2$, the inclusion of the isometry group $O(3)$ into the diffeomorphism group $\mbox{Diff}(\mathbb{S}^2)$ is a homotopy equivalence and he conjectured that the analogous result holds for $\mathbb{S}^3$, that is that $O(4)\to\mbox{Diff}(\mathbb{S}^3)$ is a homotopy equivalence. 
In 1982,  A. Hatcher \cite{H} proved the ``Smale conjecture":
\begin{thm}[Hatcher 1982 \cite{H}]\label{hatcher} The inclusion $O(4)\to \mbox{Diff}(\mathbb{S}^3)$ is a homotopy equivalence.
\end{thm}

More generally,  the (generalized) Smale conjecture asks 
\begin{conj}\label{generalized}
If $(X,g)$ is a closed orientable Riemannian three-manifold, and $g$ a metric of constant sectional curvature $\pm 1$, then the inclusion of $\mbox{Isom}(M)$ in $\mbox{Diff}(M)$ is a homotopy equivalence.
\end{conj}

For hyperbolic manifolds, Conjecture \ref{generalized} was obtained by Gabai using his insulator technology \cite{G}.   For spherical space-forms, in many cases where $M$ contains a geometrically incompressible embedded Klein bottle,  Conjecture \ref{generalized} was proved by Ivanov \cite{I}.  The remaining cases of $M$ admitting Klein bottles were obtained by McCullough-Rubinstein \cite{MR}.   If $M$ is a lens space other than $\mathbb{RP}^3$, then the conjecture was proved by Hong-Kalliongis-McCullough-Rubinstein \cite{HKMR}.  This left open the case of $\mathbb{RP}^3$ together with the spherical space-forms of Heegaard genus $2$ not admitting Klein bottles (such as the Poincar\'e homology sphere).  

For spaces $M$ with metrics of constant sectional curvature equal to $0$ (such as $M=T^3$), or reducible manifolds, the inclusion of $\mbox{Isom}(M)$ into $\mbox{Diff}(M)$ is often not surjective on connected components.  The conjecture for such spaces needs to be weakened to account for this (see Section 1 in \cite{HKMR} for a discussion).  For Haken manifolds, the correct analog of Conjecture \ref{generalized} was proved by Hatcher and Ivanov (\cite {H}, \cite{I}).

Given the complexity of Hatcher's arguments,  a long-standing problem (No.\ 30 in S.T.\ Yau's list \cite{Y}, Section 1.4 in \cite{HKMR}) has been to use methods of geometric analysis to give an analytic proof of Hatcher's result or to address the remaining cases of Conjecture \ref{generalized} (including $\mathbb{RP}^3$).  Recently both of these goals were achieved by Bamler-Kleiner \cite{BK} using Hamilton-Perelman's Ricci flow:

\begin{thm}[Bamler-Kleiner 2018 \cite{BK1},\cite{BK3}]
Conjecture \ref{generalized} is true.  
\end{thm}

Roughly speaking,  an equivalent form of Hatcher's theorem asks whether the space of round metrics on $\mathbb{S}^3$ is contractible.   Beginning with an $S^k$-family of such metrics,  one can extend to the ball $B^{k+1}$ arbitrarily and flow the entire family under Ricci flow.   Using a detailed analysis of the singularities that occur, Bamler-Kleiner obtain the desired contraction.   Their work uses Perelman's work on Geometrization together with earlier work of Kleiner-Lott \cite{KL} and Bamler-Kleiner \cite{BK1} on uniqueness of the flow through singularities.


In this paper, we instead use minimal surfaces and min-max theory to give a new proof of the Smale conjecture for all lens spaces (which includes the case of $\mathbb{RP}^3$ that was only recently obtained using Ricci flow methods).

As for the case of Hatcher's theorem,  there are several equivalent formulations of the Smale conjecture for a lens space.  The one we prove is the following.  Let $L(p,q)$ be a lens space with $p\geq 2$.
\begin{thm}[Smale Conjecture for lens spaces]\label{lensspace}
The space of Heegaard tori in round $
L(p,q)$ retracts onto the subspace of Clifford tori\footnote{We will impose certain parameterizations on these spaces. See Theorem \ref{main2} for the precise statement.}.    
\end{thm}

A \emph{Clifford torus} in a lens space is the projection of a Clifford torus in $\mathbb{S}^3$ to the lens space.   By a result of Brendle \cite{B}, Clifford tori are the only embedded minimal tori in a lens space (cf. Theorem \ref{classification}).  The connectedness of the space of Heegaard tori in lens spaces was obtained by Bonahon-Otal \cite{BO}.
 
The Smale conjecture for $\mathbb{RP}^3=L(2,1)$ most closely resembles the case of $\mathbb{S}^3$ due to both manifolds' large six-dimensional isometry group. The larger the isometry group, the more potential choice in retracting diffeomorphisms.  Its proof was expected to follow from arguments similar to those employed in Hatcher's proof for $\mathbb{S}^3$,  but perhaps due to the intricacy of \cite{H} no such argument has appeared in the intervening forty years.   In Bamler-Kleiner's work (\cite{BK1}, \cite{BK3}) for instance,  the Smale conjecture for $\mathbb{RP}^3$ was handled separately from other spherical space-forms,  due to its similarity to the $\mathbb{S}^3$ case.   

Let us briefly sketch the main ideas. 
Two important ingredients in our argument are the Lawson conjecture proved by Brendle \cite{B} and, for the case of $\mathbb{RP}^3$ and $L(4p,2p\pm 1)$, the Multiplicity One theorem recently proved by Wang-Zhou \cite{WZ}.  

Let $\mbox{Emb}(T^2,M)$ denote the space of Heegaard embeddings of tori into a round lens space $M$ and let $\mathcal{T}_{min}$ denote the space of minimal Clifford embeddings. Since $\mbox{Emb}(T^2,M)$ has the homotopy type of CW-complex, it suffices by Whitehead's theorem to show all relative homotopy groups $\pi_k(\mbox{Emb}(T^2,M), \mathcal{T}_{min})$ vanish.  Suppose instead one has a non-trivial relative homotopy class $[a]\in \pi_k(\mbox{Emb}(T^2,M), \mathcal{T}_{min})$ and let $k\geq2$ be the smallest such possible integer.  The class $[a]$ is represented by a map $f:D^k\to\mbox{Emb}(T^2,M)$ with $f|_{\partial D}$ consisting of minimal embeddings.

Using the Smale Conjecture for $\mathbb{S}^3$, we first extend $f$ to a family $\tilde{f}$ of $1$-sweepouts of $M$ (parameterized by $D^k\times[-1,1]$).  We then consider all $(k+1)$-parameter sweepouts homotopic to $\tilde{f}$ and agreeing with $\tilde{f}$ at the boundary and the corresponding min-max value for this homotopy class.  The key point is that the family is trivial,  meaning that its width (or critical value) is realized by the area of a Clifford torus equal to the maximal area of surfaces at the boundary of the $(k+1)$-sweepout.  For this step we need the resolution of the Lawson conjecture by Brendle \cite{B}. For $\mathbb{RP}^3$ and $L(4p,2p-1)$ ($p\geq1$), we additionally use the Multiplicity One Theorem of Wang-Zhou \cite{WZ} (see also earlier work of Sarnataro-Stryker \cite{DS}) to avoid getting projective planes with arbitrary even multiplicity or Klein bottles with multiplicity $2$ (respectively). Wang-Zhou's theorem
is the Simon-Smith counterpart to the Multiplicity One Theorem in the Almgren-Pitts setting \cite{Zh} obtained earlier using min-max theory for the prescribed mean curvature functional (\cite{ZZ1},\cite{ZZ2}).

By a Lusternick-Schnirelman argument (inspired by Marques-Neves' proof of the Willmore conjecture \cite{MN}), we may ``pull tight" the family $\tilde{f}$ and extract a new family $g:\Sigma^k\to\mbox{Emb}(T^2,M)$ that is \emph{homologous} to $f$, where $\Sigma^k$ is a (relative) $k$-cycle in $D^k\times[-1,1]$ with $\partial\Sigma^k=\partial D^k\times\{0\}$ whose image are embeddings all \emph{weakly close} (in the sense of varifolds) to the space of minimal tori.  
It follows  that most of the area of the surfaces in the family $g$ is contained
in suitable tubular neighborhoods around the minimal tori.
We parametrically retract small area disks that protrude outside of the 
tubular neighborhoods and then retract the family $g$ directly onto the family of Clifford tori.
Our argument uses earlier ideas of Hatcher and Ivanov (\cite{H},\cite{I}) that were originally applied in the Haken case.  This shows $g$ and $f$ are null-homologous.  By the Hurewicz theorem (as we assumed $k$ is the smallest nontrivial $\pi_k(\mbox{Emb}(T^2,M), \mathcal{T}_{min})$),  $f$ is also \emph{null-homotopic}.   This gives the desired contradiction\footnote{In the case $k=1$ (needed to apply the Hurewicz theorem), the same argument gives directly a null-homotopy rather than a null-homology.}.

We also need the computation of the oriented Goeritz groups of genus $1$ splittings of lens spaces due to Bonahon \cite{Bon} to control how the parameterizations of the family of tori may change during the above deformations.  These groups are all trivial except in the case of $\mathbb{RP}^3$.

Roughly speaking, the min-max process enables us to reduce the non-Haken case to the Haken case $T^2\times[-1,1]$ which was studied in the late 70s by Hatcher and Ivanov.   For Haken manifolds, incompressible surfaces are the canonical objects (and the Smale conjecture for such spaces amounts to making surfaces disjoint from an incompressible surface), while for non-Haken manifolds, sweep-outs and index $1$ minimal surfaces are the replacements for these objects (\cite{HKMR}, \cite{CGK}).  Unlike in the Ricci flow approach to study diffeomorphism groups, we do not need to keep track of any surgeries. 

The methods of this paper are robust with respect to higher genus Heegaard splittings. For the other spherical space-forms of genus $2$, however, there is no corresponding classification of minimal surfaces realizing the Heegaard genus. It would be natural to conjecture (analogous to Lawson's conjecture for minimal tori in $\mathbb{S}^3$, resolved by Brendle \cite{B}):

\begin{conj}
Each round spherical space-form of Heegaard genus $2$ admits exactly one embedded minimal surface of genus $2$ (up to ambient isometry).
\end{conj}


We note finally that a much-studied approach to Conjecture \ref{generalized} and Theorem \ref{hatcher} is through the study of the mean curvature flow (MCF) for submanifolds.  One equivalent form to Hatcher's theorem (Theorem \ref{hatcher}), for instance, is that the space of embedded $2$-spheres in $\mathbb{R}^3$ is contractible and one can try to use mean curvature flow through singularities to construct the contraction (see \cite{BuHaHe} for results of this kind for the moduli space of two-convex spheres).  
Recently, Bamler-Kleiner \cite{bamler2023multiplicity} proved Ilmanen's Multiplicity One Conjecture (pg 37 in \cite{Il}), which was a key missing ingredient in realizing this program. \footnote{In lens spaces, one would need to understand the MCF of families of tori for which there is also the possibility of \emph{fattening} of the flow.  White \cite{White} has conjectured that fattening does not occur for genus $1$ surfaces.}

The organization of this paper is as follows.  In Section \ref{smalelens} we study the space of tori in lens spaces  and deduce an equivalent form of the Smale conjecture.  In Section \ref{minmaxsection} we introduce the necessary min-max notions.  In Section \ref{section_main} we give the min-max argument.  In Section \ref{section_ls} we give the Lusternick-Schnirelman argument to extract the desired family  of tori close as varifolds to the space of minimal tori.  In Section \ref{section_hair} we retract the filigree or ``hair" of such embeddings to a tubular neighborhood of the space of minimal tori, and in Section \ref{section_finalretraction} we complete the retraction to the space of minimal tori.  

\begin{ack}
D.K. would like to thank Profs.\ Toby Colding, Andr\'e Neves and Fernando Cod\'a Marques for conversations.  Thanks to Prof.\ Hyam Rubinstein for encouraging discussions at the 2015 meeting held at IMPA “Hyperbolic geometry and minimal surfaces.” The authors are grateful to the anonymous referee
for many valuable comments and corrections.
\end{ack}

\section{The Smale conjecture in lens spaces}\label{smalelens}
In this section, we relate the Smale conjecture in lens spaces to a conjecture about the space of embeddings of tori in such spaces.  First we consider the space of minimal tori in such spaces, and then the Goeritz group for lens spaces. 
\subsection{Space of minimal tori in lens spaces}\label{spaceoftori}
Let us denote the round $3$-sphere
\begin{equation}
\mathbb{S}^3=\{(z,w)\in\mathbb{C}^2\;|\; |z|^2+|w|^2=1\}.
\end{equation}
For each $p\geq 2$ and $q\geq 1$ with $q<p$ and $q$ relatively prime to $p$ we consider $\mathbb{Z}^q_p$, the cyclic group of order $p$ acting on $\mathbb{S}^3$ with generator $\xi_{p,q}$:
\begin{equation}
\xi_{p,q}(z,w) = (e^{2\pi i/p} z, e^{2q\pi i/p} w).
\end{equation}

The (round) lens space $L(p,q)$ is $\mathbb{S}^3/\mathbb{Z}^q_p$ endowed with the quotient metric.  The Clifford torus 
\begin{equation}
C = \{(z,w)\;|\; |z|^2=|w|^2=\frac{1}{2}\}\subset\mathbb{S}^3 
\end{equation}
is preserved by each group $\mathbb{Z}^q_p$ and descends to an embedded minimal torus in $L(p,q)$ with area $2\pi^2/p$.

Since $L(p,q)$ and $L(q,r)$ are isometric when $r+q =p$, we will assume in this paper $1\leq q\leq p/2$.  The lens space $L(2,1)$ is $\mathbb{RP}^3$. 

Johnson-McCullough (cf. Table 2 in \cite{JM}) computed the isometry groups $\mbox{Isom}(M)$ of lens spaces (and in fact all elliptic $3$-manifolds).  Let $\mbox{isom}(M)$ denote the connected component of the identity of $\mbox{Isom}(M)$.
We have

\begin{prop}[Isometry groups of lens spaces]
Up to diffeomorphism type, there holds:
\begin{enumerate}

\item $\mbox{isom}(\mathbb{RP}^3) \cong SO(3)\times SO(3)$
\item $\mbox{isom}(L(p,1)) \cong S^1\times SO(3)$ for $p>2$ even
\item $\mbox{isom}(L(p,1)) \cong S^1\times S^3$ for $p>2$ odd
\item $\mbox{isom}(L(p,q)) \cong S^1\times S^1$ for $q\neq 1$ and $p>2$.
\end{enumerate}
\end{prop}
 
For $M$ a $3$-manifold, and $\Sigma\subset M$ an embedded surface, let $\mbox{isom}^+(M,\Sigma)$ denote the subgroup of $\mbox{isom}(M)$ taking $\Sigma$ to itself such that the induced diffeomorphism on $\Sigma$ preserves orientation.
Similarly let $\mbox{isom}(M,\Sigma)$ denote the subgroup of $\mbox{isom}(M)$ of elements mapping $\Sigma$ to $\Sigma$.  The quotient $\mbox{isom}(M)/\mbox{isom}(M,\Sigma)$ denotes the space of unoriented and unparameterized images of $\Sigma$ under $\mbox{isom}(M)$.  Similarly $\mbox{isom}(M)/\mbox{isom}^+(M,\Sigma)$ denotes the space of oriented but unparameterized images of $\Sigma$ under isometries.

We need the following classification for minimal tori in lens spaces due to Brendle \cite{B}.  Let $\pi_{p,q}:\mathbb{S}^3\to L(p,q)$ denote the projection map.  We have:
\begin{thm}\label{classification}
If $T$ is an embedded minimal torus in round $L(p,q)$ then \begin{equation}T=\pi_{p,q}(\tilde{T})\end{equation} for some Clifford torus $\tilde{T}$ in $\mathbb{S}^3$.  
\end{thm}
\begin{proof}
Lifting $T$ to $\mathbb{S}^3$, we obtain a connected minimal surface $\tilde{T}=\pi^{-1}_{p,q}(T)$ by Frankel's theorem \cite{F}.   By the multiplicativity of the Euler characteristic under covering maps,  it follows that $\tilde{T}$ is a minimal torus.  By \cite{B},  $\tilde{T}$ is a Clifford torus.
\end{proof}

For $T$ a Clifford torus embedded in a lens space $M$, let us denote $\mathcal{T}_{min}^*(M)=\mbox{isom}(M)/\mbox{isom}(M,T)$ and let $\overline{\mathcal{T}}_{min}=\mbox{isom}(M)/\mbox{isom}^+(M,T)$.  The subgroups $\mbox{isom}(M,T)$ and $\mbox{isom}^+(M,T)$ consist of a union of $2$-tori (Lemma 10.2 in \cite{JM}).

In \cite{K2} (Proposition 4.2), the space of  Clifford embeddings in lens spaces was computed:
\begin{prop}[Space of unparameterized minimal tori]\label{unparam}
The space of unoriented unparameterized Clifford tori $\mathcal{T}^*_{min}$ is given by 
\begin{enumerate}
\item $\mathcal{T}_{min}^*(\mathbb{RP}^3) \cong RP^2\times RP^2$
\item $\mathcal{T}_{min}^*(L(p,1)) \cong RP^2$ for $p>2$.
\item $|\mathcal{T}_{min}^*(L(p,q))|=1$ for $q\neq 1$ and $p>2$.
\end{enumerate}
The space $\overline{\mathcal{T}}_{min}$ is given by 
\begin{enumerate}
\item $\overline{\mathcal{T}}_{min}(\mathbb{RP}^3) \cong\tilde{ RP^2\times RP^2}$ \footnote{This space is the orientable double cover of $RP^2\times RP^2$ given by $(S^2\times S^2)/\{\pm 1\}$, with $-1\cdot (x,y) = (-x,-y)$ denoting the diagonal action.}
\item $\overline{\mathcal{T}}_{min}(L(p,1)) \cong S^2$ for $p>2$.
\item $|\overline{\mathcal{T}}_{min}(L(p,q))|=2$ for $q\neq 1$ and $p>2$.
\end{enumerate}
Moreover, the image of any minimal embedded torus in $M$ is contained in $\mathcal{T}^*_{min}$.
\end{prop}
Recall that $\mathcal{T}_{min}$ denotes the space
of parametrized minimal Clifford embeddings. Note that one can see directly from the results above that $\mathcal{T}_{min}\cong\mbox{isom}(M)$ is a $T^2$-bundle over $\mathcal{\overline{T}}_{min}(M)$  except when $M=\mathbb{RP}^3$.  
The pull-back of an element in $\mathcal{\overline{T}}_{min}(\mathbb{RP}^3)$ to $\mbox{isom}(\mathbb{RP}^3)$ consists of \emph{two} copies of $T^2$. The reason is that in addition to an element of $\mbox{isom}(\mathbb{RP}^3)$ flipping the orientation of the torus $T$ and flipping both handlebodies, there is an involutive element $\tau\in\mbox{isom}(\mathbb{RP}^3)$ taking a Clifford torus $T$ to itself, while preserving each handlebody as well as the orientation of $T$  (cf. Proposition 9.3 in \cite{JM})\footnote{Lifting to $\mathbb{S}^3$, the isometry $\tau$ acts on the Clifford torus as the hyperelliptic involution sending $(z,w)$ to $(\overline{z},\overline{w})$.}.   In fact, we have (letting $T^2$ denote the identity component of $\mbox{isom}(\mathbb{RP}^3,T)$):


\begin{equation} \label{eq_isom}
\begin{split}
\mbox{isom}(\mathbb{RP}^3)/T^2 &= (SO(3)\times SO(3))/(S^1\times S^1)\\
 & = (SO(3)/SO(2))\times (SO(3)/SO(2) ) \\
 & = S^2\times S^2,
\end{split}
\end{equation}
which is the double cover of $\overline{T}_{min}(\mathbb{RP}^3)=\tilde{RP^2\times RP^2}$, as expected.  The element $\tau$ is a generator of the appropriate Goeritz group, described in the next subsection.

\subsection{Goeritz group of lens spaces}
For $M$ a three-manifold, let $\mbox{diff}(M)$ denote the connected component of $\mbox{Diff}(M)$ that contains the identity. Given a Heegaard surface $\Sigma$ in $M$, let $\mbox{diff}(M,\Sigma)$ denote the subgroup of $\mbox{diff}(M)$ comprising the elements that take $\Sigma$ to $\Sigma$ (as a set) and preserve both handlebodies of the Heegaard splitting determined by $\Sigma$.  The (oriented) Goeritz group of the Heegaard splitting $G_0(M,\Sigma)$ is the set of connected components of $\mbox{diff}(M, \Sigma)$ in $\mbox{diff}(M)$. 

Since each element of $G_0(M,\Sigma)$ corresponds to diffeomorphisms contained in $\mbox{diff}(M)$ we may think of an element of $G_0(M,\Sigma)$ as a path of diffeomorphisms, beginning at the identity, taking $\Sigma$ to $\Sigma$ (as sets) at the end of the path.
F. Bonahon \cite{Bon} proved the following\footnote{One may also consider the Goeritz group permitting flipping orientations of $\Sigma$ in which case $L(p,1)$ for any $p\geq 2$ admit non-trivial Goeritz groups.} : 
\begin{prop}[Goeritz group of lens spaces (Proposition 9.3 \cite{JM}] \label{goeritz}Let $M$ be a lens space and $\Sigma$ a Heegaard torus.  
\begin{enumerate}
    \item $G_0(M,\Sigma)$ is trivial if $M\neq\mathbb{RP}^3$.
    \item $G_0(\mathbb{RP}^3,\Sigma)=\mathbb{Z}_2$ (generated by $\tau$).
\end{enumerate}
\end{prop}
 By \cite[Theorem 1]{gramain1973type}, \cite[p.219]{FaMa} we have that (up to homotopy type)
\begin{equation}
\mbox{Diff}(T^2) \simeq T^2\times\mbox{GL}(2,\mathbb{Z}), 
\end{equation}
and the element $\tau\in G_0(\mathbb{RP}^3,\Sigma)$ maps the component corresponding to the identity $I\in\mbox{GL}(2,\mathbb{Z})$ to $-I\in\mbox{GL}(2,\mathbb{Z})$.  

\subsection{Equivalent form of Smale Conjecture in lens spaces}
Let $M$ be a lens space and $T^2$ a $2$-torus.   Let $\mbox{Emb}(T^2, M)$ denote the space of smooth embeddings from $T^2$ to $M$.  We endow these spaces with the $C^\infty$ Fr\'echet topology. By a result of Palais, $\mbox{Emb}(T^2, M)$ and $\mbox{Diff}(M)$ have the homotopy type of CW-complexes (cf. \cite{P} or the discussion in \cite{HKMR}).  

The space $\mbox{diff}(M)$ maps to $\mbox{Emb}(T^2, M)$ in the following way.  First fix a reference Clifford torus embedding $f_0:T^2\to M$ and let \begin{equation}\mbox{Emb}_0(T, M)\subset \mbox{Emb}(T^2, M)\end{equation} denote the connected component of $f_0$.
Then define the map
\begin{equation}\label{action}
p_2: \mbox{diff}(M)\to \mbox{Emb}_0(T^2, M)
\end{equation}
by 
\begin{equation}
p_2(\Phi)=\phi\circ f_0.
\end{equation} 
By the isotopy extension theorem, the map \eqref{action} is surjective.

By a result of Palais \cite{P}, $p_2$ is a fibration.   The fiber $F_0=p_2^{-1}(f_0)$ consists of diffeomorphisms of $M$ that fix the torus $f_0(T^2)$ pointwise.  If $M$ is not flippable, then any element of $F_0$ also preserves each component of $M\setminus f_0(T^2)$.  In the case that $M$ is flippable (i.e.  any lens space of the form $L(p,1)$ for $p\geq 2$),  it is also true that any element of $F_0$ does not interchange the two handlebodies $M\setminus f_0(T^2)$.   Indeed, suppose $d\in F_0$ flips the handlebodies but fixes $f_0(T^2)$ pointwise.  Let $\{E_1, E_2, E_3\}$ be an oriented basis of $T_xM$ at $x\in f_0(T^2)$ where $E_1$ and $E_2$ are chosen to be in the tangent plane of the torus $f_0(T^2)$. Since $d$ fixes $f_0(T^2)$ pointwise, it follows that the pushforwards satisfy $d^*_x(E_1) = E_1$ and $d^*_x(E_2)=E_2$.  On the other hand, since $d$ swaps the handlebodies we get $d^*_x(E_3) = -\lambda^2 E_3+\alpha E_1+\beta E_2$ for some $\lambda\neq 0$.  The diffeomorphisms $d$ induces the oriented basis $\{E_1, E_2, -\lambda^2 E_3+\alpha E_1+\beta E_2\}$ of $T_xM$.  Since the determinant of this change-of-basis is negative, it follows that $d$ reverses orientation which contradicts the fact that $d$ is contained in the identity component of $\mbox{Diff}(M)$.

Thus $F_0$ consists of diffeomorphisms that fix each handlebody $M\setminus f_0(T^2)$ (set-wise) while leaving the respective boundaries fixed (point-wise).   By Hatcher's theorem on Haken manifolds\footnote{The result was originally proved in the PL category, which was upgraded to the smooth category using Hatcher's proof of the Smale conjecture.} (Section 1 in \cite{H2} and item (3) in the list of equivalent forms in the appendix of \cite{H}): \begin{thm} \label{first}
    Let $N^3$ be a handlebody.
   Then $\mbox{Diff}_p(N, \partial N)\simeq \{*\}$.\footnote{Throughout this paper $\mbox{Diff}_p(M,\partial M)$ denotes the diffeomorphisms of $M$ that fix $\partial M$ pointwise.} 
\end{thm}

 By Theorem \ref{first},  $F_0$ is homotopy equivalent to the product of two contractible spaces, and thus is contractible.   
 From the long exact sequence for a fibration we get that $p_2$ induces isomorphisms on all homotopy groups.  By Whitehead's theorem,  $p_2$ is a homotopy equivalence.  

Let $p_1$ be the restriction of $p_2$ to $\mbox{isom}(M)$.   Let $\mbox{Emb}_{min}(T^2,M)$ denote the image of the reference embedding $f_0$ under $p_1$.   We have  

\begin{lemma} $p_1:\mbox{isom}(M)\to\mbox{Emb}_{min}(T^2,M)$ is a homeomorphism.\label{equal2}\end{lemma}
\begin{proof}  Surjectivity follows by construction.  Let us show that $p_1$ is injective.   If $p_1(I_2)= p_1(I_1)$ then the isometry $I_1^{-1}I_2$ fixes $T$ pointwise.  But the fixed point set of an isometry is totally geodesic,  which the Clifford torus is not.  Thus $I_1^{-1}I_2$ is the identity element in $\mbox{isom}(M)$ and we get $I_1=I_2$.   \end{proof}

It follows from the previous discussion that we have the following commutative diagram:
\[
\begin{tikzcd}\label{commutative}
\mbox{isom}(M) \arrow[r, "\iota_1"] \arrow[d, "p_1"]
& \mbox{diff}(M) \arrow[d, "p_2" ] \\
\mbox{Emb}_{min}(T^2,M) \arrow[r, "\iota_2"]
& |[rotate=0]| \mbox{Emb}_0(T^2,M)
\end{tikzcd}
\]

In other words,\begin{equation}
p_2:(\mbox{diff}(M),\mbox{isom}(M))\to (\mbox{Emb}_0(T^2,M),\mbox{Emb}_{min}(T^2,M))
\end{equation}
is a homotopy equivalence of pairs. 
By Theorem \ref{basichomotopy} we obtain that $p_2$ induces an isomorphism
\begin{equation}\label{is}
\pi_k(\mbox{diff}(M),\mbox{isom}(M))\cong\pi_k(\mbox{Emb}_0(T^2,M),\mbox{Emb}_{min}(T^2,M))
\end{equation}

For ease of notation, let us denote
\begin{equation}
\mathcal{T} :=\mathcal{T}(M) = \mbox{Emb}_0(T^2,M), 
\end{equation}
and recall
\begin{equation}
\mathcal{T}_{min}:=\mathcal{T}_{min}(M)=\mbox{Emb}_{min}(T^2,M).
\end{equation}

Thus \eqref{is} becomes
\begin{equation}\label{areequiv}
\pi_k(\mbox{diff}(M),\mbox{isom}(M))\cong\pi_k(\mathcal{T},\mathcal{T}_{min}).
\end{equation}

By Proposition \ref{unparam}, $\mathcal{T}_{min}$ coincides with the space of minimal embedded tori in $M$ (with certain parameterizations), justifying the terminology.  

We can now show:

\begin{thm}\label{main2}
The Smale conjecture for a lens space is equivalent to the statement that the inclusion of $\mathcal{T}_{min}(M)$ in $\mathcal{T}(M)$ is a homotopy equivalence.  Equivalently,
\begin{equation}\label{relative2}
\pi_k(\mathcal{T},\mathcal{T}_{min}) = 0.
\end{equation}
for each $k\geq 0$.  
\end{thm}
\begin{proof}
For any lens space, the inclusion of the isometry group in the diffeomorphism group is a bijection on path components (\cite{Bon}).  Thus the Smale conjecture (Conjecture \ref{generalized}) for $M$ is the assertion that $\iota_1$ is a homotopy equivalence.  By Whitehead's theorem it suffices to show that $\iota_1$ induces isomorphisms on all homotopy groups.     By the long exact sequence for pairs
\begin{equation}
\to\pi_k(\mbox{isom}(M))\to\pi_k(\mbox{diff}(M))\to \pi_k(\mbox{diff}(M),\mbox{isom}(M))\to \pi_{k-1}(\mbox{isom}(M))\to
\end{equation}
we see that $\iota_1$ induces isomorphisms if and only if $\pi_k(\mbox{diff}(M),\mbox{isom}(M))$ vanishes for all $k\geq 0$.   In light of \eqref{areequiv}, this is equivalent to \eqref{relative2}.  
\end{proof}

In light of Theorem \ref{main2} the main result of this paper is
\begin{thm}[Smale Conjecture for lens spaces]\label{main}
\begin{equation}\label{relative}
\pi_k(\mathcal{T},\mathcal{T}_{min}) = 0.
\end{equation} 
for each $k\geq 0$.
\end{thm}
The case $k=0$ in \eqref{relative} follows from Bonahon-Otal \cite{BO}.

\section{Min-max theorem}\label{minmaxsection}
In this section, we  introduce the necessary min-max notions that we will need.    Throughout this paper, $\mathcal{H}^2(A)$ denotes the two-dimensional Hausdorff measure (``area") of a set.  

Set $I^k=[-1,1]^k\subset \mathbb{R}^k$.  Let $\{\Sigma_t\}_{t\in I^k}$ be a family of closed subsets of $M$ and $B\subset\partial I^k$.  We call the family $\{\Sigma_t\}_{t\in I^k}$ a genus $g$ sweepout if 
\begin{enumerate}
\item $\mathcal{H}^2(\Sigma_t)$ is a continuous function of $t\in I^k$
\item $\Sigma_t$ converges to $\Sigma_{t_0}$ in the Hausdorff topology as $t\to t_0$ .
\item For $t_0\in I^k\setminus B$, $\Sigma_{t_0}$ is a closed surface of genus $g$ and $\Sigma_t$ varies continuously 
with $t$ in the smooth topology.  
\item For $t\in B$, $\Sigma_t$ consists of a $1$-complex.  
\end{enumerate}

 Let $\{\Sigma_t\}_{t\in I^k}$ be a genus $g$ sweepout.    We will now describe the competitor sweepouts.   Let $F: I^k\to \mbox{diff}(M)$ be a continuous map with $F|_{\partial I^k} = id$ and expressed so that $D_F(*,t):M\to M$ is the diffeomorphism corresponding to $F(t)$.  Assume further that $F$ is contractible rel $\partial I^k$.  Then $\{D_F(\Sigma_t,t)\}_{t\in I^k}$ is also a genus $g$ sweepout.   Denote by $\Pi$ the family of sweepouts obtained from this procedure.  We say $\Pi$ consists of all sweepouts \emph{homotopic to $\{\Sigma_t\}_{t\in I^k}$ relative to $\partial I^k$}. Alternatively, we call $\Pi$ the \emph{saturation} of the family $\{\Sigma_t\}_{t\in I^k}$.  

We define the min-max width as
\begin{equation}
\omega(\Pi) = \inf_{\chi\in\Pi}\sup_{t\in I^k}  \mathcal{H}^2(\chi_t).
\end{equation}
A \emph{minimizing sequence} is a sequence of sweepouts $\{\Sigma^i_t\}_{t\in I^k}\in \Pi$ with
\begin{equation}
\lim_{i\to\infty} \sup_{t\in I^k} \mathcal{H}^2(\Sigma^i_t)= \omega(\Pi).
\end{equation}
Given a minimizing sequence,  a \emph{min-max sequence} is a sequence $\{\Sigma_{t_i}^i\}_{i=1}^\infty$ (for some $t_i\in I^k$) with the property that
\begin{equation}
\lim_{i\to\infty} \mathcal{H}^2(\Sigma^i_{t_i})=\omega(\Pi).  
\end{equation}
The main point of the min-max theory is that if the homotopy class $\Pi$ is non-trivial,  then some min-max sequence converges in a weak sense to a minimal surface.  In particular we have the following: 
\begin{thm}[Min-max Theorem]\label{minmax}
If 
\begin{equation}
\omega(\Pi) > \sup_{t\in\partial I^k}\mathcal{H}^2(\Sigma_t)
\end{equation} 
then there exists a minimizing sequence $\chi_i\in \Pi$ and min-max sequence $\{\chi_i(t_i)\}_{i=1}^\infty$ so that
\begin{equation}\label{howitconverges}
\chi_i(t_i)\to \sum_{i=1}^L n_i\Gamma_i\mbox{ as varifolds}
\end{equation}
where $\{\Gamma_i\}_{i=1,2,..,L}$ are a family of smooth embedded pairwise-disjoint minimal surfaces and $n_1,...,n_L$ are positive integers.  

Moreover, 
\begin{equation}\label{widthrealized}
\omega(\Pi) = \sum_{i=1}^L n_i\mbox{Area}(\Gamma_i).
\end{equation}

If $\mathcal{O}$ (resp. $\mathcal{N}$) denotes the subcollection of $i$ so that $\Gamma_i$ is orientable (resp.  non-orientable),  then we have the genus bounds with multiplicity
\begin{equation}\label{genusbounds}
\sum_{i\in\mathcal{O}} n_ig(\Gamma_i)+\frac{1}{2}\sum_{i\in\mathcal{N}} (n_i-1) g(\Gamma_i)\leq g.
\end{equation}
If we assume in addition that $\{\Sigma\}_{t\in I^k}$ are separating embeddings\footnote{Note that in this case, obtained by Wang-Zhou \cite{WZ}, the minimal surfaces arise from a different procedure, by perturbing the area functional and also perturbing the metric to be bumpy. Thus we do not obtain the minimizing and min-max sequence in precisely the same way.}, then there exist a positive integer $L'$, minimal surfaces $\{\Gamma'_i\}_{i=1,2,..,L'}$ and positive integers $n'_1,...,n'_{L'}$ as above, \eqref{genusbounds} and \eqref{widthrealized} hold (with $L'$ in place of $L$, $\Gamma'_i$ in place of $\Gamma_i$ and $n'_i$ in place of $n_i$) and in addition 
\begin{enumerate}
\item  If $\Gamma'_i$ is two-sided and unstable,  then $n'_i=1$. 
\item If $\Gamma'_i$ is one-sided,  then the connected double cover of $\Gamma'_i$ is stable.  
\end{enumerate}
\end{thm}
The existence result is due to Simon-Smith \cite{SS} (see \cite{CD} for an exposition when $k=1$, and the appendix of \cite{CGK} for the extension to several parameters).  The latter two statements under the assumption that the sweepouts are separating were proved recently by Wang-Zhou (Theorem 7.2 in \cite{WZ}).  The genus bound \eqref{genusbounds} is due to the first-named author \cite{KG} (weaker genus bounds without the $n_i$ on the LHS were obtained by Simon-Smith \cite{SS} (cf. \cite{DP}).  

\begin{rmk}In the Simon-Smith theory one often forgets the parameterizations of the original sweepout generating the homotopy class $\Pi$ since the weak varifold convergence in \eqref{howitconverges} is not in the sense of parameterized maps. It will be important for us, however, to keep track of the parameterizations of our original family.  In both the pull-tight procedure and verification of the almost minimizing property, competitor sweepouts are constructed via ambient isotopies, which can be applied to parameterized maps as well as unparameterized.\end{rmk}

We will use the following fact about lens spaces:
\begin{thm}[Width of lens spaces]\label{lens}
Let $M$ be a round lens space and $\Pi$ the saturation of the genus $1$ Heegaard sweepout of $M$.   Then 
\begin{equation}\label{widthoneparam}
\omega(\Pi) = \frac{2\pi^2}{p}.
\end{equation}
Moreover, the only smooth minimal surface realizing the width is a (multiplicity $1$) Clifford torus.
\end{thm}
\begin{proof}
By \cite{KMN} the width of $L(p,q)$ is achieved by a minimal torus $\Gamma$ of index $1$.   By Theorem \ref{classification} the result follows.  
\end{proof}

\section{Proof of Theorem \ref{main}}\label{section_main}
We prove $\pi_k(\mathcal{T},\mathcal{T}_{,min})=0$ for each $k\geq 1$ by induction on $k$.   Let us describe the inductive step.  Suppose $k\geq 2$ and $\pi_i(\mathcal{T},\mathcal{T}_{,min})=0$ for all $i<k$.   By the relative Hurewicz theorem (cf. Theorem 4.37 in  \cite{Ha}),  there is a natural map
\begin{equation}
H: \pi_k(\mathcal{T},\mathcal{T}_{min})\to H_k(\mathcal{T},\mathcal{T}_{min})
\end{equation}
descending to a bijection $H':\pi_k(\mathcal{T},\mathcal{T}_{min})/\pi_1(\mathcal{T}_{min})\to H_k(\mathcal{T},\mathcal{T}_{min})$.   
We will show:
\begin{claim}\label{mainclaim}
For each $k\geq 1$ and any $a\in \pi_k(\mathcal{T},\mathcal{T}_{min})$,  $H(a)=0$ in $H_k(\mathcal{T},\mathcal{T}_{min})$.   
\end{claim}

 Assuming Claim \ref{mainclaim},  since $H'$ is a bijection,  and the orbit $[e]\in\pi_k(\mathcal{T},\mathcal{T}_{min})/\pi_1(\mathcal{T}_{min})$ containing the identity $e$ maps under $H'$ to the trivial element,  it follows that the orbit $[e]$ consists of all elements in $\pi_k(\mathcal{T},\mathcal{T}_{min})$.   Thus any element in $\pi_k(\mathcal{T},\mathcal{T}_{min})$ can be realized by acting an element of $\pi_1(\mathcal{T}_{min})$ on $e$.   It follows that $\pi_k(\mathcal{T},\mathcal{T}_{min})$ is trivial.    This completes the inductive step.

We will also show the base case of the induction (which is a stronger form of Claim \ref{mainclaim}):
\begin{claim}\label{mainclaim1}
$\pi_1(\mathcal{T},\mathcal{T}_{min})$ is trivial.  
\end{claim}

Let us show Claim \ref{mainclaim} (henceforth we will assume $k\geq 1$,  and keep track to show that when $k=1$ the stronger result Claim \ref{mainclaim1} holds).    The homotopy class $a\in \pi_k(\mathcal{T},\mathcal{T}_{min})$ is represented by a sweepout \begin{equation}\Phi': I^k\to\mathcal{T}\end{equation}   
with
\begin{equation}
\Phi'|_{\partial I^k} \subset\mathcal{T}_{min}.
\end{equation}
Let us assume $a$ is based at $f'_0\in\mathcal{T}_{min}$ so that (writing $I^k=I^{k-1}\times I$)
\begin{equation}
\Phi'_{(\partial I^{k-1}\times I) \cup I^k\times\{-1\}}=f_0'.
\end{equation}
Recall
\begin{equation}
p_2: (\mbox{diff}(M), \mbox{isom}(M))\to (\mathcal{T},\mathcal{T}_{min}), 
\end{equation}
is a homotopy equivalence of pairs.  

By Theorem \ref{basichomotopy} there exists $q:(\mathcal{T},\mathcal{T}_{min})\to (\mbox{diff},\mbox{isom})$,  a homotopy inverse of pairs to $p_2$.   Thus there is a homotopy $P_t:\mathcal{T}\to\mathcal{T}$ with $P_0=\mbox{id}$,  $P_1=p_2\circ q$ and satisfying $P_t(\mathcal{T}_{min})\subset \mathcal{T}_{min}$ for all $t\in [0,1]$.   

Set $\Phi = p_2\circ q\circ\Phi'$.   Then $\Phi\in \pi_k(\mathcal{T},\mathcal{T}_{min})$ which is now based at $f_0=p_2\circ q(f'_0)\in\mathcal{T}_{min}$.    If $H(\Phi)$ is trivial,  since $\{P_t(\Phi)\}_{t\in [0,1]}$ gives a homotopy between $\Phi$ and $\Phi'$, that induces isomorphisms on homology,  it follows that $H(\Phi')$ is also trivial.    Thus it suffices to show the triviality of $H(\Phi)$ to prove Claim \ref{mainclaim}.

The family of diffeomorphisms $b=q(\Phi')$ (which projects to $\Phi$ under $p_2$) permits us to extend the sweepout $\Phi:I^k\to\mathcal{T}$ to a $(k+1)$-parameter sweepout of $M$ in the following way. 

Let $\{f_t\}_{t\in (-1,1)}$ be a family of embeddings $f_t:T^2\to M$ defined as follows.  Fix a choice of normal vector $\mathbf{n}$ on $f_0(T^2)$.   For each $t\in [-1,1]$,  define the map
\begin{equation}
g_t:T^2\to M
\end{equation}
obtained by $g_t(x) = \exp_x(\frac{\pi t\mathbf{n}(x)}{4})$.    Then set $f_t = g_t\circ f_0$.   As $t\to\pm 1$, the image of $f_t$ converges to an unknotted circle (the spines of the respective handlebodies)\footnote{The distance between a Clifford torus in $\mathbb{S}^3$ and either of its core geodesics is $\pi/4$, which descends to the lens spaces.}.  Set $\Sigma_t = f_t(T^2)$.  Then $\{\Sigma_t\}_{t\in [-1,1]}$ is an optimal $1$-sweepout of the lens space $M$ by Theorem \ref{lens}.  In fact, $\{\Sigma_t\}_{t\in [-1,1]}$ is an optimal foliation, meaning that in addition to being an optimal sweepout, $\{\Sigma_t\}_{t\in (-1,1)}$ foliates $M\setminus (\Sigma_{-1}\cup\Sigma_1)$.

Recall 
\begin{equation}
b:I^k\to \mbox{diff}(M), 
\end{equation}
and let us define the $(k+1)$-parameter sweepout of $M$ by genus $1$ surfaces (parameterized by $I^{k+1}=I^k\times[-1,1]$) by first defining the family of embeddings
\begin{equation}
\Phi(x,t):=b(x)\circ f_t.
\end{equation}
Let us denote the correspoding $(k+1)$-sweepout
\begin{equation}
\tilde{\Phi}(x,t) = \Phi(x,t)(T^2)
\end{equation}

The sweepout $\{\tilde{\Phi}(x,t)\}_{(x,t)\in I^{k+1}}$ extending $\Phi$ satisfies the following properties:
\begin{enumerate}
\item $\tilde{\Phi}(x,0)=\Phi(x)(T^2)$ for each $x\in I^k$.
\item $\{\tilde{\Phi}(x,t)\}_{t\in [-1,1]}$ is a $1$-sweepout of $M$ for each $x\in I^k$ and an optimal foliation for $x\in\partial I^k$.
\item $\tilde{\Phi}(x,0)$ is the image of an element in $\mathcal{T}_{min}$ for all $x\in\partial I^k$.  
\item $\tilde{\Phi}(x,-1)\mbox{ and } \tilde{\Phi}(x,1)$ consist of embedded circles for each $x\in I^k$.  
\end{enumerate}
Let $\Pi=\Pi(\tilde{\Phi})$ denote the saturation of the family $\tilde{\Phi}$ and the corresponding min-max value
\begin{equation}
\omega(a):= \inf_{\Psi\in\Pi} \sup_{(x,t)\in I^k\times[-1,1]}\mathcal{H}^2(\Psi(x,t)). 
\end{equation}

\noindent
By item (1) we get that 
\begin{equation}
\frac{2\pi^2}{p}=\sup_{(x,t)\in \partial (I^k\times [-1,1])}\mathcal{H}^2(\tilde{\Phi}(x,t)).
\end{equation}
Thus if
 \begin{equation}\label{width}
\omega(a)>\frac{2\pi^2}{p}\end{equation} 
it follows from the Min-Max theorem \ref{minmax} that there exists a minimizing sequence of sweepouts $\Phi_i\in\Pi$ and corresponding min-max sequence $\Phi_i(x_i,t_i)$ so that 
\begin{equation}
\Phi_i(x_i,t_i)\to k\Gamma\mbox{ as varifolds}
\end{equation}
where $k$ is a positive integer and $\Gamma$ is a connected\footnote{This follows by Frankel's theorem \cite{F} since $M$ has positive Ricci curvature.},  embedded minimal surface.  Furthermore,  we have
\begin{equation}
\omega(a)=k\mathcal{H}^2(\Gamma).
\end{equation}

By the second part of Theorem \ref{minmax}, there exists a two-sided orientable minimal surface $\Gamma'\subset M$ with genus at most $1$ satisfying
\begin{equation}
\omega(a) = \mathcal{H}^2(\Gamma').
\end{equation}

Note that $\Gamma'$ is not a minimal two-sphere (as otherwise, lifting to $\mathbb{S}^3$ we obtain disjoint minimal two-spheres, which do not exist).  Thus by Theorem \ref{classification}, $\Gamma'$ is a Clifford torus.  This violates \eqref{width}.\footnote{Note that Wang-Zhou's Multiplicity $1$ result is needed in those lens spaces which admit embedded Klein bottles or projective planes ($L(4p,2p-1)$ for $p\geq1$ and $\mathbb{RP}^3$, respectively), as in the other lens spaces the genus bound \eqref{genusbounds} is sufficient to arrive at the same contradiction.}


Since we have reached a contradiction assuming \eqref{width}, we get
\begin{equation}
\omega(a) = \frac{2\pi^2}{p}.
\end{equation}

We conclude \emph{a posteriori} from Theorem \ref{classification} and the genus bounds \eqref{genusbounds} that $k=1$ and $\Gamma$ is a Clifford torus. 

Choose a minimizing sequence $\Phi_i\in\Pi$ so that for some sequence $\epsilon_i\to 0$ there holds
\begin{equation}\label{ls}
\frac{2\pi^2}{p}\leq \sup_{(x,t)\in (I^k\times[-1,1])} \mathcal{H}^2(\Phi_i(x,t)) \leq \frac{2\pi^2}{p}+\epsilon_i.  
\end{equation}
Since the sweepout $\{\Phi_i(x,t)\}_{(x,t)\in I^k\times I}\in \Pi$ arose from applying suitable diffeomorphisms (coming from the saturation of $\Pi$) to $\{\Phi(x,t)(T^2)\}_{(x,t)\in I^k\times I}$, by slight abuse of notation, we will consider $\Phi_i(x,t)$ as a family of parameterized maps from $T^2$ into $M$ rather than their image surfaces (or circles).  With this understanding, for each $i$
\begin{equation}
\Phi_i|_{I^k\times \{0\}}\subset\mathcal{T}.
\end{equation}

By construction each family $\{\Phi_i(x,0)\}_{x\in I^k}$ is homotopic (rel $\partial I^k$) to the initial family $\{\Phi(x,0)\}_{x\in I^k}= \{\Phi(x)\}_{x\in I^k}$.  

If $\alpha$ is equal to the $k$-chain $I^k\times\{0\}$ then  \begin{equation}
[\alpha]\in H_k(I^k\times [-1,1], \partial I^k\times\{0\})\cong\mathbb{Z}
\end{equation}   
represents a generator. Since (relative) homotopic maps induce homologous relative cycles (if we denote by $(\Phi_i)_*$ the pushforward of $\Phi_i$ in homology):
\begin{equation}
[\Phi_*(\alpha)]=[(\Phi_i)_*(\alpha)]\in H_k(\mathcal{T}, \mathcal{T}_{min}).
\end{equation}
Note that $H(a) = [(\Phi_i)_*(\alpha)]$.   To show Claim \ref{mainclaim},  it remains to demonstrate:
\begin{thm}\label{equality}
$[(\Phi_i)_*(\alpha)] = [0] \mbox{ in } H_k(\mathcal{T},\mathcal{T}_{min})\mbox{ for large } i.$
\end{thm}


The proof of Theorem \ref{equality} will proceed in three steps.   First, we show using a Lusternik-Schnirelman argument that we may find a representative of the homology class $[H(a)]$ with all corresponding surfaces close in the sense of varifolds to the space of Clifford tori.  In the second,  inspired by Hatcher's arguments we show that 
 we may retract all the ``filigree" or ``hair" of this family so that each torus is contained in some small tubular neighborhood $T^2\times[-r,r]$ around a minimal torus.  In the third,  we use Hatcher-Ivanov's work on Haken manifolds (applied to the tubular neighborhood $T^2\times[-r,r]$) to retract these surfaces to the space of minimal tori.

\section{Lusternik-Schnirelman argument}\label{section_ls}
Following \cite{MN} we let $\mathbf{F}$ denote the metric on the space of integral $2$-currents in $M$ obtained as the sum of the flat metric and $\mathbf{F}$-metric for induced varifolds.
\begin{equation} \label{def of F}
     \F(V,W) = \mathcal{F}(V-W) + \F(|V|,|W|)
\end{equation}
Recall that $\overline{\mathcal{T}}_{min}$ denotes the set of embedded,  oriented minimal tori in $M$.

In the setting of the previous section,  we show:
\begin{prop}[Retraction to $\mathbf{F}$-metric tubular neighbohood]\label{lsargument}
Fix $\epsilon>0$.  For $i$ large enough there exists a relative cubical $k$-cycle $\alpha_i$ with  \begin{equation}[\alpha]=[\alpha_i]\in H_k(I^k\times [-1,1], \partial I^k\times\{0\})\cong\mathbb{Z}\end{equation} with $\partial\alpha_i = \partial I^k\times\{0\}$ so that pushing $[\alpha_i]$ forward via $\Phi_i$:
\begin{equation}
\Phi_i: I^k\times[-1,1]\to\mathcal{T},
\end{equation}
there holds
\begin{equation}
[H(a)]=[(\Phi_i)_*(\alpha_i)]\in H_k(\mathcal{T},\mathcal{T}_{min})
\end{equation}
Furthermore, for any $(x,t)\in supp(\alpha_i)$ there holds
\begin{equation}
\mathbf{F}(\Phi_i(x,t), \overline{\mathcal{T}}_{min})\leq \epsilon.
\end{equation}
\end{prop}

The proof of Theorem \ref{lsargument} is analogous to Theorem 9.1 in Marques-Neves' proof of the Willmore conjecture \cite{MN}.   We follow their notation (cf. Section 7 in \cite{MN}) which we introduce for the reader's convenience.

Let us denote $I^k=[-1,1]^k\subset \mathbb{R}^k$.   For each $j\in\mathbb{N}$,  $I(1,j)$ denotes the cell complex on $I^1$ whose $1$-cells are $[-1,-1+2^{-j}], [-1+2^{-j},-1+2\cdot 2^{-j}],...[1-2^{-j},1]$ and whose zero cells are $[-1],[-1+2^{-j}],...,[1]$.   Set $\mbox{dim}(\beta)=1$ if $\beta$ is a $1$-cell, and $0$ if $\beta$ is a $0$-cell (vertex).  
We denote by $I(k,j)$ the $k$-dimensional cell complex on $I^k$ given by 
\begin{equation}
I(k,j) = I(1,j)\otimes ...\otimes I(1,j)\mbox{ ($k$ times).}
\end{equation}
For each $0\leq p\leq k$ we say \begin{equation}\alpha=\alpha_1\otimes\alpha_2....\otimes \alpha_k\end{equation} is a \emph{$p$-cell} if and only if $\alpha_i$ is a cell of $I(1,j)$ for each $i$ and $\sum_{i=1}^k\mbox{dim}(\alpha_i) = p.$  For each $0\leq p\leq n$ let $I(k,i)_p$ denote the subcollection of $p$-cells.

The boundary homomorphism $\partial: I(k,j)\to I(k,j)$ is given by 
\begin{equation}\label{boundarymap}
\partial(\alpha_1\otimes...\otimes\alpha_k) =\sum_{i=1}^k (-1)^{\sigma(i)}\alpha_1\otimes...\otimes \partial\alpha_i\otimes...\otimes\alpha_k
\end{equation}
with 
\begin{equation}
\sigma(i) =\sum_{p<i} \mbox{dim}(\alpha_p), 
\end{equation}
and 
\begin{equation}
\partial([a,b])=[b]-[a] \mbox{ for } [a,b]\in I_1(1,j), 
\end{equation}
as well as
\begin{equation}
\partial([a]) = 0\mbox{ for } [a]\in I_0(1,j).
\end{equation}

We also need the following elementary lemma relating the distance in $\mathbf{F}$-metric to the distance in Hausdorff topology for the optimal foliation $\{\Sigma_t\}_{t\in [-1,1]}$ defined in Section \ref{section_main}. 
\begin{lemma}\label{fversushausdorff}
For each $\delta>0$ there exists $g(\delta)>0$ so that 
\begin{enumerate}
\item For any $t$, if $\mathbf{F}(\Sigma_t, \Sigma_0)\leq \delta$,  then $\mathbf{F}(\Sigma_s,\Sigma_0)\leq g(\delta)$ for all $s$ between $0$ and $t$ \label{eq1}
\item $g(\delta)\to 0$ as $\delta\to 0$.  
\end{enumerate}
\end{lemma}
\begin{proof}
For each $\delta>0$ set $g(\delta)$ to be the infimal $g(\delta)$ so that item (1) holds for all $t\in [-1,1]$.  Since the diameter of the family $\{\Sigma_t\}_{t\in [-1,1]}$ in the $\mathbf{F}$-metric is finite,  $g(\delta)$ is finite.   Suppose (2) fails.  Then we get a sequence of $\delta_i\to 0$ with $g(\delta_i)>\eta$ for some $\eta>0$.   For each $\delta_i$,  choosing $t_i$ close enough to the infimal case (assuming without loss of generality $t_i>0$) we have 
\begin{equation}
\mathbf{F}(\Sigma_{t_i},\Sigma_0)\leq\delta_i \mbox{ and } g(\delta_i)/2\leq\mathbf{F}(\Sigma_{s_i},\Sigma_0)\mbox{ for some } 0\leq s_i\leq t_i.
\end{equation}
Since $\delta_i\to 0$ the fact that $\mathbf{F}(\Sigma_{t_i},\Sigma_0)\leq\delta_i$ implies $t_i\to 0$, and thus also $s_i\to 0$.  Thus $\mathbf{F}(\Sigma_{s_i},\Sigma_0)\geq \eta/2$ for $s_i\to 0$,  which is a contradiction.  
\end{proof}


Let us return to the proof of Proposition \ref{lsargument}.  
\begin{proof}

Fix $\delta>0$ (we will specify it precisely at the end of the proof).  Consider the set $U^i_\delta\in I^{k+1}$ given by

 \begin{equation}
U^i_\delta=\{(x,t)\in I^{k+1}\;|\; \mathbf{F}(\Phi_i(x, t),\mathcal{\overline{T}}_{min})\leq \delta\}.
\end{equation}

Assume $\delta$ is so small so that $I^k\times\{-1\}$ and $I^k\times\{1\}$ are disjoint from $U^i_\delta$.   We claim that for all $i$ large enough,  the bottom face $I^k\times\{-1\}$  and the top face $I^k\times\{1\}$ are contained in different connected components of $I^{k+1}\setminus U^i_\delta$.  
 
Otherwise for some fixed $\delta$ we obtain a subsequence of $i$ and paths $\gamma_{i}(s)=(x_i(s),t_i(s))$, $\gamma_i: [0,1] \to I^{k+1}$, with\begin{enumerate}
\item $\gamma_{i}(0)\subset I^k\times\{-1\}$ 
\item $\gamma_{i}(1)\subset I^k\times\{1\}$.
\item $\mathbf{F}(\Phi_i(\gamma_{i}(s)),\mathcal{\overline{T}}_{min})\geq\delta$ for each $s\in[0,1]$.  \label{faraway}
\end{enumerate}
Since $\{\Phi_i(\gamma_{i}(s))\}_{s\in [0,1]}$ is a $1$-sweepout,  it follows from \eqref{ls} and Theorem \ref{lens} that $\{\Phi_i(\gamma_{i}(s)\}_{s\in [0,1]}$ is a minimizing sequence for the $1$-parameter min-max problem. Thus by the Min-max Theorem \ref{minmax}, some min-max sequence $\{\Phi_i(\gamma_{i}(s_i)) \}$ converges in the sense of varifolds to a smooth embedded minimal surface $\Gamma$ (multiplicity greater than $1$ is excluded). By Theorem \ref{lens}, $\Gamma$ is a Clifford torus.  The convergence with multiplicity one implies that the corresponding sequence of flat cycles converges with no loss of mass in the limit.  By \cite[Proposition A.1]{de2013existence} this implies that $\{\Phi_i(\gamma_{i}(s_i) )\}$ also converges to a Clifford torus in the 
$\mathbf{F}$-metric. In light of (\ref{faraway}), this gives a contradiction.



Let $B_\delta^i$ denote the component of $I^k\times[-1,1]\setminus U^i_\delta$ containing $I^k\times\{-1\}$.   Consider $I^k\times[-1,1]$ as a cubical complex $I(k+1, j)$, where $j$ is chosen so large so that for each $(k+1)$-cell $\Delta$ there holds
\begin{equation}\label{fineness}
\mathbf{F}(\Phi_i(x),\Phi_i(y))\leq \delta/2, 
\end{equation}
for any $x, y\in \Delta$ (where $x$ and $y$ denote a $k+1$ tuple of points in $I^1$).
Note that the choice of the fineness $j$ depends on $i$.   

Let $C^i_\delta$ denote the $(k+1)$-chain in $I(k+1,j)$ consisting of all $(k+1)$-cells which contain at least one point of $B_\delta^i$.   Then we have by the triangle inequality for all $x\in C^i_\delta$
\begin{equation}
 \mathbf{F}(\Phi_i(x),\mathcal{\overline{T}}_{min})\geq \delta/2.   
\end{equation}

Let $b(i)$ denote the set of $k$ cells in $I(k+1,j)$ that are faces of exactly one $k+1$-cell in $C^i_\delta$.   Then
\begin{equation}
\partial C^I_\delta =\sum_{\alpha\in b(i)}\mbox{sgn}(\alpha)\alpha, 
\end{equation}
where $\mbox{sgn}$ of a cell is either $1$ or $-1$.   Note that all cells comprising the bottom face $I^k\times\{-1\}$ are contained in $\partial C^i_\delta$ and by the definition of the boundary map \eqref{boundarymap} the sign of such a cell is $(-1)^{k+1}$.  

By increasing $i$ if necessary,   every cell of $\partial C^i_\delta$ is disjoint from the ``top face" $k$-chain $I^k\times \{1\}$.  Otherwise we obtain a subsequence of $i$ and corresponding paths $\gamma_i$ from bottom to top face contained in $U^i_{\delta/2}$ with all corresponding surfaces an $\mathbf{F}$-distance at least $\delta/2$ from the space of Clifford tori $\mathcal{\overline{T}}_{min}$  giving the same contradiction as above.    


Let us define the $k$-chain:
\begin{equation}
R(i) = (-1)^k \partial C_\delta^i + (-1)^{k-1} \partial I^k \times [-1,0] + I^k\times\{-1\}.
\end{equation}
By construction the $k$-chain $R(i)$ is disjoint from the bottom face $I^k\times \{-1\}$.  We may also compute 
\begin{equation}\label{rel}
\partial R(i) = (-1)^k \partial^2C_\delta^i +\partial I^k\times \partial[-1,0]+ \partial I^k\times\{-1\} = \partial I^k\times\{0\}.
\end{equation}

From \eqref{rel} it follows that the $k$-chain $R(i)$ is an element of the relative homology group $H_k(I^k\times[-1,1], \partial I^k\times\{0\},\mathbb{Z})$.  Note that $\partial I^k\times\{0\}$ is a $(k-1)$-sphere.  From the long exact sequence for relative homology groups the boundary map: 
\begin{equation}
\partial: H_k(I^k\times[-1,1], \partial I^k\times\{0\},\mathbb{Z}) \to H_{k-1}(\partial I^k\times\{0\},\mathbb{Z})\cong \mathbb{Z}
\end{equation}
is a well-defined isomorphism.   Thus since $\partial R(i)=\partial I^k\times\{0\}$ and $\partial (I^k\times\{0\}) = \partial I^k\times\{0\}$, it follows that 
\begin{equation}
[R(i)]=[I^k\times\{0\}]\in H_k(I^k\times[-1,1], \partial I^k\times\{0\},\mathbb{Z}) .
\end{equation}

We now claim 
\begin{claim}\label{isclose}
For any $x\in R(i)$ there holds
\begin{equation}
\mathbf{F}(\Phi_i(x), \mathcal{\overline{T}}_{min})\leq \frac{3\delta}{2}+g(2\delta), 
\end{equation}
where $g(\delta)\to 0$ as $\delta\to 0$.  
\end{claim}

There are two cases to consider.  First assume $\Delta$ is a $k$-cell in $R(i)$ that is not contained in $\partial (I^k\times [-1,1])$.   Let us denote the totality of such interior $k$-cells by $R(i)_0$.  Then by definition of $b(i)$ we have that $\Delta$ is the face of a $(k+1)$-cell that is not contained entirely in $U_\delta^i$.   Let $y$ be a point in this $(k+1)$-cell that is not contained in $U_\delta^i$.   Then by the triangle inequality (and the fineness condition \eqref{fineness}) we get for any $x\in\Delta$:
\begin{equation}\label{interior}
\mathbf{F}(\Phi_i(x),\mathcal{\overline{T}}_{min})\leq \mathbf{F}(\Phi_i(x),  \Phi_i(y)) +  \mathbf{F}(\Phi_i(y),\mathcal{\overline{T}}_{min})\leq \frac{3\delta}{2}.
\end{equation}

We now consider the $k$-cells in $R(i)_b$ of $R(i)$ that are contained in the boundary $\partial I^k\times [-1,1]$.  Let $s^i_{min}$ and $s^i_{max}$ denote the minimal (resp. maximal) value of $s$ so that $supp( R(i))\cap (\partial I^k\times\{s\})\neq\emptyset$.  

Let $\Delta\subset R(i)_b$ be a cell that when expanded $\alpha_1\otimes...\otimes\alpha_{k+1}$ has that $\alpha_{k+1}$ is a $1$-cell $[s_1,s_2]$ with $s_2=s_{max}^i$ (or with $s_1=s_{min}^i$).   Then since $R(i)$ is a relative cycle, it follows that $\Delta$ shares a $(k-1)$-face with a cell $\Delta'\in R(i)_0$.   Thus by the triangle inequality and \eqref{interior} applied to $\Delta'$ we obtain for any $y\in\Delta$
\begin{equation}
\mathbf{F}(\Phi_i(y),\mathcal{\overline{T}}_{min})\leq 2\delta.
\end{equation}

Applying Lemma \ref{fversushausdorff} we then obtain for \emph{any} cell $\Delta''\in R(i)_b$ and $x\in\Delta''$
 \begin{equation}\label{boundarypart}\mathbf{F}(\Phi_i(x), \mathcal{\overline{T}}_{min})\leq g(2\delta),\end{equation} 
where $g(7\delta/2)\to 0$ as $\delta\to 0$.  

The equations \eqref{interior} and \eqref{boundarypart} together complete the proof of Claim \ref{isclose}.  Finally we choose $\delta$ small enough so that \begin{equation}\frac{3\delta}{2}+g(2\delta)<\epsilon\end{equation} and set $\alpha_i=R(i)$.

\end{proof}

\subsection{The case $k=1$}
If $k=1$,  Theorem \ref{lsargument} furnishes a continuous path $\gamma(t)$ in $I\times [-1,1]$ from the midpoint of the left side of the parameter space $\{0\}\times\{0\}$ to the midpoint of the right side $\{1\}\times\{0\}$.   Such a path is homotopic to the horizontal path $\eta(t): I\times[-1,1]$ given by $\eta(t) = (t,0)$ for $t\in [-1,1]$.    Thus we obtain a stronger conclusion:
\begin{thm}[$k=1$ case]\label{k=1case}
For all $\epsilon>0$ small enough,  there exists a map $a':[0,1]\to\mathcal{T}$ homotopic (rel $\partial [0,1]$) to the representative $a\in\pi_1(\mathcal{T},\mathcal{T}_{min})$, with the property that 
\begin{equation}
\mathbf{F}(a'(t),\overline{\mathcal{T}}_{min})\leq\epsilon,
\end{equation}
for all $t\in [0,1]$.  
\end{thm}

\section{Retracting filigree of tori}\label{section_hair}

In this section, inspired by ideas of Hatcher-Ivanov (\cite{H2}, \cite{I2}) we deform a family
of embeddings that is close in the varifold topology to 
$\overline{\mathcal{T}}_{min}$ to a family that is close to $\overline{\mathcal{T}}_{min}$ in the Hausdorff topology. This is accomplished by parametrically retracting
filigree\footnote{Filigree is ornamental wiring.  To the authors' knowledge, the first example of retracting such parts of surfaces in a geometric context appears in work of Almgren-Simon \cite{AS} on the existence of embedded solutions to Plateau's problem in convex balls.} or ``hair'' that protrudes outside of the small tubular neighborhood of a minimal surface. We formulate
the result in a more general setting, replacing the torus with any Heegaard surface $\Sigma$.

In this section we will assume that space $X$ is equipped with finite cubical decomposition and let $\partial X$ denote the support of the boundary of the cubical complex 
defined by (\ref{boundarymap}).
For $\Sigma$ an orientable surface embedded in a $3$-manifold, denote by $N_r(\Sigma)$ the $r$-tubular neighborhood about $\Sigma$.  For $r$ sufficiently small, $N_r(\Sigma)$ is diffeomorphic to $\Sigma\times [-r,r]$. 

\begin{prop} \label{hair}
      Let $\Sigma \subset M$ be a Heegaard surface in an orientable Riemannian $3$-manifold $M$
    and let $X$ denote a cubical complex.  For $r$ small enough, there exists $\varepsilon>0$ with the following property.
    Suppose $f: X \rightarrow Emb(\Sigma, M^3)$ is a continuous map
    with 
    \begin{enumerate}
    \item $\F(f(x), \Sigma)< \varepsilon$  for all $x \in X$ and 
    \item $f(x) \subset N_{\frac{r}{5}}(\Sigma)$ for all $x \in \partial X$.
    \end{enumerate}
Then there exists a homotopy
    $F: [0,1] \times X \rightarrow Emb(\Sigma, M^3)$, such that
    \begin{enumerate}
        \item $F(0,x) = f(x)$ for all $x \in X$;
        \item $F(t,x) = f(x)$ for $x \in \partial X$;
        \item $F(1,x)$ is contained in $N_r(\Sigma)$ for all $x \in X$.
    \end{enumerate}
\end{prop}

\begin{proof}
Suppose $r_0>0$ is chosen small enough so that $N_{r_0}(\Sigma)$ is diffeomorphic to $\Sigma\times [-r_0,r_0]$. 
Let $i_0>0$ denote the infimal injectivity radius of the smoothly varying surfaces $\{\partial N_s(\Sigma)\}_{s\in [0,r_0]}$.  Choose $r<r_0$ small enough so that $r<i_0$.

Let $p$ denote the projection map $p:N_r(\Sigma)\to\Sigma$. Shrinking $r$ if necessary, there exists $\Lambda = \Lambda(M)>0$ so that for any smooth closed surface $\Gamma\subset N_r(\Sigma)$,
\begin{equation}\label{areadistortion}
|\mbox{Area}(\Gamma)-||p_{\#}(\Gamma)|||\leq \Lambda r^2.
\end{equation}
Here $||p_{\#}(\Gamma)||$ denotes the total mass of the projected varifold $p_{\#}(\Gamma)$, which counts overlaps with their respective multiplicity (cf. page 7 in \cite{CD}).  
Denote $f(x)$ by $\Sigma_x$.
Assume that $\varepsilon>0 $ is small enough
so that \begin{equation}\label{2} \mathcal{H}^2(\Sigma_x \setminus N_{\frac{r}{10}}(\Sigma)) <A_0 < \frac{r^2}{100}.\end{equation}

Given $x \in X$, by the coarea formula we have
\begin{equation} \label{1}
\int _{r/4}^{r/2}\mathcal{H}^1(\partial N_s(\Sigma) \cap \Sigma_x) ds\leq\mathcal{H}^2(\Sigma_x \cap (N_{\frac{r}{2}}(\Sigma) \setminus  N_{\frac{r}{4}}(\Sigma)))< \frac{r^2}{100}.\end{equation}
Hence, by Sard's lemma, \eqref{2} and \eqref{1} we can find $\tilde{s}(x) \in (\frac{r}{4},\frac{r}{2})$,
such that $\Sigma_{x'}$ intersects $\partial N_{\tilde{s}(x)}(\Sigma)$ transversally for all $x'$
in a small neighborhood of $x$, and
 the length of $L(x') = \Sigma_{x'} \cap \partial N_{\tilde{s}(x)}(\Sigma)$
is at most $l_0< r$.  By the choice of $r$, we get that each component of $L(x')$ is a circle bounding a disk in $\partial N_{\tilde{s}(x)}(\Sigma)$ and by the (uniform) isoperimetric inequality for the surfaces $\{\partial N_s(\Sigma)\}_{s\in [0,r]}$ there exists $C=C(M)>0$ so that the area of all such disks is at most $Cr^2$. 

By compactness there exists a finite set of balls $\mathcal{B} = \{B_i\}$,
such that balls of $\frac{1}{10}$ radius cover $X$, $X \subset \bigcup \frac{1}{10} B_i$,
and there exists $s_i \in (\frac{r}{4},\frac{r}{2})$, such that $\Sigma_x$ intersects 
$\partial N_{s_i}(\Sigma)$ transversally with the length of $L(x)$ bounded by $ r$ for all $x \in B_i$.
(If the ball $B_i$ intersects $\partial X$ we set $s_i = \frac{r}{4}$; note that we can 
arrange for balls $B_i \in \mathcal{B}$ that intersect $\partial X$ to be sufficiently small, so that the intersection $\Sigma_x \cap \partial N_{\frac{r}{4}}(\Sigma)$
is empty for $x \in B_i$.) Note that we can also assume that
for all $B_i$'s that do not intersect $\partial X$ the values of $s_i$ are distinct.

Fix a cubical subdivision of $X$ and let $X^{(k)}$ denote the $k$-skeleton of the subdivision.  The cubulation is chosen so that the diameter of each cell is less than
$\min_i \{ \frac{1}{10} rad(B_i) \}$. For each  $v \in X^{(0)}$
there exists a ball $B(v)=B_i \in \mathcal{B}$, such that $\Delta \subset B(v)$
for each cell $\Delta$ that contains $v$. Let $P_v = \partial  N_{s_i}(\Sigma)$
denote the corresponding surface that intersects $\Sigma_v$ transversally
and set $s(v) = s_i$. 
Note that by construction $\Sigma_x \cap P_v$ is transverse and has controlled length for all $x$ in a cell 
that contains $v$.  Fix a cell $\Delta$ and let $v \in \Delta \cap X^{(0)}$ be a vertex.  For $x \in \Delta$ we have that $L(x) = P_v \cap \Sigma_x$ is a smooth family 
of circles. 

We claim that there exists $E>0$ so that each connected component of $L(x)$ bounds a disk in $\Sigma_x$
of area at most $Er^2$. Let us first show that each component of $L(x)$ bounds a disk. Indeed, for the length $l_0$ of $L(x)$ sufficiently small
we can perform ``neck-pinching'' surgeries along $L(x)$ (cf.\ Lemma 4.3 in \cite{KeLiSo})
and obtain a closed surface $\Sigma_x'$ in the tubular neighborhood of $\Sigma$
that still coincides with $\Sigma$ on a large set.  Then the projection of $N_r(\Sigma)$ onto $\Sigma$ restricted to $\Sigma_x'$
is a degree $1$ map. If a connected component of $L(x)$ did not bound a disk,
then $\Sigma_x'$ has a smaller genus than $\Sigma$, which is impossible.  
Similarly, one can deduce from closeness in varifold topology that
the areas of any disk in $\Sigma_x$ bounded by $L(x)$ must be small (depending on $r$).  More precisely, by \eqref{areadistortion} and the choice of $C$ we have (letting $D$ denote the union of all disks in $\Sigma_x$ bounded by the curves $L(x)$)
\begin{equation}\label{area3}
\mbox{Area}(\Sigma)-\Lambda r^2\leq \mbox{Area}(\Sigma_x')= \mbox{Area}(\Sigma_x)-\mbox{Area}(D)+Cr^2.
\end{equation}
In the inequality we have used that 
$||p_{\#}(\Sigma_x')||\geq \mbox{Area}(\Sigma)$ 
together with \eqref{areadistortion}. 
By shrinking $\varepsilon$ we can arrange
\begin{equation}\label{area4}
|\mbox{Area}(\Sigma)-\mbox{Area}(\Sigma_x)|\leq \varepsilon\leq r^2.  
\end{equation}
Thus we obtain for some $E>0$ (combining \eqref{area3} and \eqref{area4}):
\begin{equation}\label{areaofdisks}
\mbox{Area}(D)\leq Er^2.  
\end{equation}

We will say that a disk $D$ in $\Sigma_x$ has ``small area" if it satisfies inequality (\ref{areaofdisks}).
We have that each connected component of $L(x)=P_v \cap \Sigma_x$ bounds a disk of small area and we let
 $D^v_x$ denote the union 
of all such disks. Define $C^v_x = \partial D^v_x$. 
We can think of $C^v_x$ as the union of ``outermost'' circles and observe that it has the following properties:

\begin{itemize}
    \item each connected component $c$ of $C^v_x$ bounds 
a disk $D_c$ of small area, such that $D_c$ is disjoint from $C^v_x \setminus c$;
    \item  the union
$D^v_x$
of disks $D_c$ for all connected components $c$ of $C^v_x$ satisfies $P_v \cap \Sigma_x \subset D^v_x$.
\end{itemize}


The following key property follows immediately from the above definition.

\begin{lemma} \label{nested}
    Suppose $v,w$ are vertices of $\Delta$ and $s(v) < s(w)$, then $D^w_x \subset D^v_x$
    for all $x \in \Delta$.
\end{lemma}


Given a vertex $v \in X$ let $\eta(v)>0$ be such that the intersection of $\Sigma_x$
with $\partial N_{s(v)+\eta'}(\Sigma)$
is transverse and has length $<l$ for all $x \in B(v)$
and $\eta' \in (-\eta(v), \eta(v))$. Moreover, 
it will be convenient to assume $\eta(v) < \frac{|s(v')-s(v)|}{5}$ for all vertices $v' \neq v \in X^{(0)} \setminus \partial X$.

For a cell $\Delta$ of $X \times [0,1]$ let $\overline{\Delta}$ denote the 
projection of $\Delta$ onto $X \times \{ 0 \}$ and let $s(\Delta) = \min \{ s(v): v \in \overline{\Delta}^{(0)} \}$
and let $v(\Delta)$ be the vertex of $\overline{\Delta}$ with $s(v(\Delta))=s(\Delta)$.

We will define the family $\Sigma_{(x,t)}$ inductively
on the $k$-skeleton of $X \times [0,1]$, so that for each $k$-cell $\Delta$ there 
exists an extension of the family $\{D^{v(\Delta)}_{(x,t)} \}_{(x,t) \in \partial \Delta}$
    to a smooth family of disks $\{D^{v(\Delta)}_{(x,t)} \}_{(x,t) \in \Delta}$ in $\Sigma_{(x,t)}$
with the following properties:

\begin{enumerate}
    \item $\Sigma_{(x,t)} \cap P_{v(\Delta)} \subset D^{v(\Delta)}_{(x,t)} $;
    \item There exists $\eta_\Delta \in (0, 2^{dim(\Delta)-dim(X)}\eta(v))$, so that for all $(x,t) \in \Delta$ there holds \begin{equation}\Sigma_{(x,t)} \setminus N_{\eta_\Delta}(D^{v(\Delta)}_{(x,t)}) =
    \Sigma_{(x,0)} \setminus N_{\eta_\Delta} (D^{v(\Delta)}_{(x,0)})\end{equation} \label{boundary of D}
    (here $N_{\eta_\Delta}$ denotes the tubular neigbourhood within
    $\Sigma_{(x,t)}$);
    \item 
   $\Sigma_{(x,1)} \subset N_{r}(\Sigma)$
    for all $(x,1) \in \Delta$;
    \item $D^{v(\Delta)}_{(x,t)}  = D^{v(\Delta)}_{(x,0)}$ for $x \in \Delta \cap (\partial X \times [0,1])$.
\end{enumerate}

First we describe the construction for $X^{(0)} \times [0,1]$.
For a vertex $v \in X_0 \setminus \partial X$ we define an isotopy of $D^v_v$ that deforms $D^v_v$
to a collection of disks in a very small neighborhood of $P_v$. 
This isotopy will be constant on $\Sigma_v \setminus N_{\eta}(D^v_v)$.

The deformation of $D^v_v$ is performed by induction on the number of connected components of $D^v_v \cap P_v$. Let $c$
be an innermost circle of $D^v_v \cap P_v$
in some disk in $D^v_v$ and let $D \subset D^v_v$ be a disk bounded by $c$. 
Since the length of $c$ is small, as observed earlier, it bounds a disk $D'$ of area at most $Cr^2$ in $P_v$.
Thus by \eqref{areaofdisks} the surface $D\cup D'$ is a 2-sphere of area at most $(C+E)r^2$ and, shrinking $r$  if necessary, $D\cup D'$ can be guaranteed to bound a $3$-ball $B \subset M$
(cf. Lemma 1 in \cite{MSY}).
By Alexander's theorem there exists a smooth ambient isotopy that deforms $D$ to $D'$
and is constant outside of a small neighborhood of $B$.  Iterating this process 
gives
the extension of $F$ to $X^{(0)} \times [0,1]$.

So far we have defined the extension over all $0$-cells and the vertical $1$-cells.  Now we prove the inductive step.  Suppose the extension has been defined over all $k$-cells.  Let $\Delta$ be a $(k+1)$-cell of $X \times [0,1]$ 
that does not lie in $X \times \{0\}$.  By the inductive assumption we have defined $\Sigma_{(x,t)}$ on $(x,t) \in \partial \Delta$, so that properties
(1)-(4) are satisfied for each cell comprising $\partial \Delta$.
By Lemma \ref{nested} and property (\ref{boundary of D}) we have that the family of disks $D^{v(\Delta)}_{(x,t)} \subset \Sigma_{(x,t)}$ is well-defined for $(x,t) \in \partial \Delta$.

We now consider two cases. 

Case 1. $\Delta = \Delta' \times [0,1]$ for a cell $\Delta'$ of $X$.

If $\Delta' \subset \partial X$ we set $\Sigma_{(x,t)} = \Sigma_{(x,0)}$.  Assume the interior of $\Delta'$
is disjoint from $\partial X$. Let \begin{equation}\Sigma_{(x,t)}' = \Sigma_{(x,t)} \setminus N_{\frac{3}{2}\eta_{\Delta'}}(D^{v(\Delta')}_{(x,t)}),\end{equation}
where $\eta_{\Delta'}$ is determined from item (2) in the induction applied to the cell $\Delta'$.
Observe that by our choice of $\eta_{\Delta'}$, the surface $\Sigma'_{(x,t)}$ is obtained from $\Sigma_{(x,t)}$ topologically by removing finitely many disks.  By the isotopy extension theorem there exists a family of diffeomorphisms 
$\{ \Phi_x \}_{x \in  \Delta'} \subset\mbox{diff}(M)$, satisfying 
$\Phi_x \circ f(x) = f(v(\Delta'))$.

By property (\ref{boundary of D}) we have $\Sigma_{(x,t)}' = \Sigma_{(x,0)}'$
and hence $\Phi_{x}(\Sigma_{(x,t)}') = \Sigma_{(v(\Delta),0)}'$ for all $(x,t) \in \partial \Delta$.
Consider the $3$-manifold with boundary \begin{equation} M' = M \setminus N_{\eta'}(\Sigma_{(v(\Delta),0)}'),\end{equation}
where $\eta' \in (0, \frac{\eta_{\Delta'}}{2})$ 
is sufficiently small so that \begin{equation}D_{(x,t)} = \Phi_{x}(\Sigma_{(x,t)}) \cap M'\end{equation}
is a smooth family of disks with fixed boundary in $\partial M'$
for $(x,t) \in \partial \Delta$. Observe that if $H$ is a genus $g$ Heegaard surface from which $k$ disks have been removed, then the boundary of a small tubular neighborhood of $H$ is a genus $2g+k-1$ Heegaard surface.  As $M'$ arises from such a procedure, $M'$ is a handlebody.  By construction, the family $D_{(x,t)}$ consists of compressing disks for $M'$.  

It will be convenient to set $y=(x,t) \in \partial \Delta $.
We will define a deformation $\{D_{(y,s)} \}_{(y,s) \in \partial \Delta \times [0,1]}$ of the family of embedded disks
$\{D_{(y,0)}=D_y \}$ to a constant family $\{D_{(y,1)}=D_v \}$.
We proceed by induction on the number $k$ of connected components
of $D_y$. 
Let $D_y^1$ denote a connected component of $D_y$.
By Theorem 1 (b) in  \cite{H3} there exists a contraction  
$\{ D_{(y,s)}^1 \}_{(y,s) \in \partial \Delta \times [0,1]}$
of embedded disks  $\{ D_y^1 = D_{(y,0)}^1 \}$ 
with $D_{(y,1)}^1 = D_v^1$ for all $y \in \partial \Delta$.
By the isotopy extension theorem there exists a family of ambient isotopies 
$\Psi_{(y,s)}: M' \rightarrow M'$ that fix $\partial M'$
and satisfying $\Psi_{(y,s)}(D_v^1) = D_{(y,s)}^1$. 
Then $\{ \Psi_{(y,s)} \circ \Psi_{(y,0)}^{-1} (D_y) \}_{(y,s) \in \partial \Delta \times [0,1]}$ 
is an isotopy that starts on $\{ D_y\}_{y \in \partial \Delta}$ and 
contracts the connected component $D_y^1$ to disk $D_v^1$.
Proceeding this way we can contract all connected components of $D_y$. Specifically,
we redefine $M'$ to exclude the small tubular neighborhood of disks that we already
deformed to a constant disk and apply 
\cite[Theorem 1(b)]{H3} and the isotopy extension theorem as described above.

The deformation described above gives an extension $\{ D_{(x,t)}\}$ for all $(x,t) \in \Delta$.
We then define \begin{equation}\Sigma_{(x,t)} = (\Sigma_{(x,0)} \setminus \Phi_{x}^{-1}(M')) \cup \Phi_{x}^{-1}(D_{(x,t)} )\end{equation} for $(x,t)$ in the interior of $\Delta$.  Taking $\eta_\Delta:= \eta'$, this completes the inductive step in this case.  

Case 2. $\Delta \subset X \times \{1\}$. 

In this case we proceed exactly
as in Case 1, except that in place of $M$ we put $N_r(\Sigma)$.  
To extend $F$ to a $1$-dimensional cell $E$ in $X^{(1)} \times \{ 1\} $ with $\partial E = \{ (w,1), (v(E),1) \}$
we use Alexander's theorem and an ambient isotopy $\Phi$ that fixes $\Sigma_x \setminus D^{v(E)}_x$
to contract $D^{v(E)}_x$ inside $N_r(\Sigma)$.

For higher dimensional skeleta we construct $M' \subset N_r(\Sigma)$ as in Case 1 (using $N_r(\Sigma)$ in place of $M$). Note that $M'$ is a compression body, and the disks $D_{(x,t)}$ are then compressing disks for this compression body.  In particular, Theorem 1 (b) in  \cite{H3} applies in this 
setting to give a contraction of families of disks in $M'$. In the case when $\Sigma \cong S^2$ we observe that
small area implies that
all disks must lie in the same isotopy class and so the arguments of \cite{H3} apply to give the desired contraction.
\end{proof}

    Let $\Sigma$ be a Heegaard surface in $M$.
Let $\mbox{isom}^+(M,\Sigma)$ denote the subgroup of $\mbox{isom}(M)$ consisting of isometries $\phi \in \mbox{isom}(M)$, such that
$\phi(\Sigma)=\Sigma$ and $\phi$ induces the same orientation on $\Sigma$.
Let $\overline{\mathcal{T}}(\Sigma,M) = \mbox{isom}(M) /\mbox{isom}^+(M,\Sigma)$. We can think of the finite dimensional manifold $\overline{\mathcal{T}}(\Sigma,M)$\footnote{It is possibly a disconnected manifold as when $M=\mathbb{RP}^3$ is endowed with the round metric and $\Sigma$ is a Heegaard torus.} as the space of
unparametrized oriented surfaces obtained from $\Sigma$ by an ambient isometry.

We will need the following two lemmas:

\begin{lemma} \label{contractibility}
    There exists a monotonically increasing contractibility function $$\rho: \mathbb{R}_{\geq 0} \rightarrow \mathbb{R}_{\geq 0}$$ with $$\lim_{r \rightarrow 0} \rho(r) =0,$$ such that for every $S \in \overline{\mathcal{T}}(\Sigma,M) $
    the set
    $B^{\F}_r(S) \cap \overline{\mathcal{T}}(\Sigma,M)$ is contractible in 
    $B^{\F}_{\rho(r)}(S) \cap \overline{\mathcal{T}}(\Sigma,M)$. 
\end{lemma}

\begin{proof}
    Fix a Riemannian metric $g$ on $\overline{\mathcal{T}}(\Sigma,M)$. As in the proof of Lemma \ref{fversushausdorff} we have that for any $r'>0$, if $r>0$ is small enough, 
    $B^{\F}_r(S) \cap \overline{\mathcal{T}}(\Sigma,M) \subset B_{r'}^g(S)$. Let $\rho(r)>0$ be such that 
    $B_{r'}^g(S) \subset B^{\F}_{\rho(r)}(S) \cap \overline{\mathcal{T}}(\Sigma,M)$. By the definiton of the $\F$
    metric we can choose it so that $\rho(r) \rightarrow 0$ as $r \rightarrow 0$. Then for $r'$ 
    less than the injectivity radius of $(\overline{\mathcal{T}}(\Sigma,M),g)$
    we have that $B^{\F}_r(S) \cap \overline{\mathcal{T}}(\Sigma,M)$ is contractible in 
    $B^{\F}_{\rho(r)}(S) \cap \overline{\mathcal{T}}(\Sigma,M)$.
\end{proof}


\begin{lemma} \label{g}
    For every $\varepsilon>0$ there exists $\varepsilon' >0$ with the following property.
    Suppose $f: X \rightarrow Emb(\Sigma, M)$
    is a continuous map, such that for every $x \in X$ we have $\F(f(x), \overline{\mathcal{T}}(\Sigma,M))< \varepsilon'$. Then there exists a    
    continuous
    map $g: X \rightarrow \overline{\mathcal{T}}(\Sigma,M) $ with $\F(f(x), g(x)) < \varepsilon$.

    Moreover, for every $x\in\partial X$, $g(x)$ coincides with the non-parametrized oriented
    torus induced by $f(x)$.
\end{lemma}

\begin{proof}
    Pick a fine cubical subdivision of $X$, 
    so that for each cell $\Delta$ of the subdivision  
    we have $\F(f(x),f(y))< \varepsilon'$ for all $x,y \in \Delta$.
    Let $n$ be the maximal dimension of a face in $X$ 
    and assume $\varepsilon'>0$ to be chosen small enough, so that $\varepsilon'< \frac{\varepsilon}{10}$ and $(\rho')^{n}(3\varepsilon')= \rho' \circ ... \circ \rho'(3\varepsilon')<\frac{\varepsilon}{3}$,
    where $\rho'(r) = \rho(r)+ 3 \varepsilon'$ and $\rho$ is the contractibility function
    from Lemma \ref{contractibility}.

    Define map $g$ on the $0$-skeleton of $X$
    to $\overline{\mathcal{T}}$ by mapping it to a closest surface in $\overline{\mathcal{T}}$ in the $\F$ metric. Note that by the triangle inequality we have $\F(g(x),g(y)) \leq 3 \varepsilon' $ for $\{x,y\} = \partial E$
    and $E$ a $1$-cell in $X$.

    Assume, by induction, that we have extended $g$ to the $(k-1)$-skeleton of $X$, so that for each $k$-cell $\Delta$ of $X$ we have $g(\partial \Delta) \subset B^{\F}_{(\rho')^{k-1}(3\varepsilon')}(v)$ for every vertex $v \in \Delta$.
    Then by the contractibility property 
    (\ref{contractibility}) we can extend $g$ to $\Delta$ by mapping it to the contraction of $g(\partial \Delta)$ inside $B^{\F}_{\rho((\rho')^{k-1}(3\varepsilon'))}(v)$. This gives an 
    extension to the $k$-skeleton of $X$. By the triangle 
    inequality for every $(k+1)$-cell $\Delta'$ and every vertex $v$ of $\Delta'$
    we have $\partial \Delta' \subset B^{\F}_{\rho((\rho')^{k-1}(3\varepsilon'))+3\varepsilon'}(v)
    \subset B^{\F}_{(\rho')^k(3\varepsilon')}(v)$. This finishes the construction.

    Property $\F(f(x), g(x)) < \varepsilon$ follows by the triangle inequality.
\end{proof}

%

Assuming the same set up as in Lemma \ref{g} we have the following parametric version of ``contracting hair''
Proposition \ref{hair}:

\begin{prop} \label{hair_parametric}
For every $r>0$, there is $\varepsilon >0$ with the following property.  Suppose
\begin{equation}
F:X\to Emb(\Sigma,M),
\end{equation}
satisfies
\begin{equation}\label{isclose2}
\mathbf{F}(F(x),\overline{\mathcal{T}}(\Sigma,M))<\varepsilon.
\end{equation}
Suppose further $F(\partial X)$ consists of embeddings supported in $\overline{\mathcal{T}}(\Sigma,M)$.
Then there exists a continuous map $g:X\to\overline{\mathcal{T}}(\Sigma,M)$ and homotopy $F: [0,1] \times X  \rightarrow Emb(\Sigma,M)$, such that
\begin{enumerate}
 \item $F(0,x) = f(x)$,
 \item $F(1,x) \subset N_r(g(x))$
 \item $F( \star,x)$ is constant for all $x \in \partial X$.
\end{enumerate}
Moreover, for every $x\in\partial X$, $g(x)$ coincides with the non-parametrized oriented
    torus induced by $f(x)$.
\end{prop}

\begin{proof}
Let map $g$ be as in Lemma \ref{g}. 

Observe that
for every $\delta>0$ there exists $\rho(\delta)$ (which may depend on $F$) with the
following property: For every $x_0 \in X$ there exists a family
of isometries $\Phi_{x_0}: B_{\rho(\delta)}(x_0) \rightarrow\mbox{isom}(M) $ with  $\Phi_{x_0}(B_{\rho(\delta)}(x_0)) \subset B_{\delta}^{C^2}(id)$ of $\mbox{isom}(M)$, such that $\Phi_{x_0}(x)(g(x)) = g(x_0)$.

We now proceed in a way very similar to the proof of Proposition \ref{hair}.
As in that proof, we will construct a family of contractions
by induction on the dimension of the skeleton of a fine subdivision of $X$.
The key difference is that to define the deformation for each small cell of the subdivision we
will first apply a family of isometries $\Phi_v(x)$ (that are close to the identity)
that rotate minimal tori $g(x)$, so that they coincide with some fixed 
minimal torus $g(v)$. After the ``hair contraction'' is defined as in Proposition \ref{hair}
we rotate the family back by applying $\Phi_v(x)^{-1}$.

By compactness we can choose a subdivision of $X$, 
so that for each vertex $v$ there exists a family of surfaces $P_v(x) = \partial N_{s(v)}(g(x))$ with the property that
$f(x)$ intersects $P_v(x)$ transversally
for all $x \in C$, whenever $C$ is a cell of $X$
containing $v$.  As in the proof of Proposition \ref{hair}, by \eqref{isclose2} we may additionally assume that each circle in $f(x)\cap P_v(x)$ bounds a disk in both $f(x)$ and $P_v(x)$.

Let $\{ \Sigma_{(x,0)} \}$ denote the family of surfaces $\{ f(x) \}$.
We will define the desired extension of this family to $\Sigma_{(x,t)} = F(x,t)$, $(x,t) \in X \times [0,1]$.
As in the proof of Proposition \ref{hair},
for each cell $\Delta$ of $X \times [0,1]$ we will inductively define a family 
of surfaces $\Sigma_{(x,t)}$
and a family of embedded disks 
$D^{v(\Delta)}_{(x,t)} \subset \Sigma_{(x,t)}$ continuous in the smooth topology with the following properties:

\begin{enumerate}
    \item $\Sigma_{(x,t)} \cap P_{v(\Delta)}(x) \subset D^{v(\Delta)}_{(x,t)} $;
    \item There exists $\eta_\Delta \in (0, 2^{dim(\Delta)-dim(X)}\eta(v))$, so that for all $(x,t) \in \Delta$ there holds \begin{equation}\Sigma_{(x,t)} \setminus N_{\eta_\Delta}(D^{v(\Delta)}_{(x,t)}) =
    \Sigma_{(x,0)} \setminus N_{\eta_\Delta} (D^{v(\Delta)}_{(x,0)}).\end{equation} 
    \item 
   $\Sigma_{(x,1)} \subset N_{r}(g(x))$
    for all $(x,1) \in \Delta$;
    \item $D^{v(\Delta)}_{(x,t)}  = D^{v(\Delta)}_{(x,0)}$ for $(x,0) \in \Delta \cap (\partial X \times [0,1])$.
\end{enumerate}

As in Proposition \ref{hair}, $v(\Delta)$ is defined as the vertex with 
minimal value of $s(v)$ among all vertices of the projection of $\Delta$
onto $X \times \{ 0\}$.

For $X^{(0)} \times [0,1]$ defining the family of disks as above
amounts to contracting ``hair'' $\Sigma_{(v,0)} \setminus N_{s(v)}(g(v))$
exactly as in the proof of Proposition \ref{hair}.

Fix a cell $C$ and assume that $\Sigma_{(x,t)}$
and $D^{v(C_i)}_{(x,t)}$ were defined for each
cell $C_i$ in the boundary of $C$ and
$(x,t) \in C_i$. Consider a family of surfaces
$$\tilde{\Sigma}_{(x,t)} =\Phi_{v(C)}(x)( \Sigma_{(x,t)}) $$
for $(x,t) \in \partial C$. Exactly as in the proof of Proposition \ref{hair} we can then define a
family of disks $\{ \tilde{D}^{v(C)}_{(x,t)} \}_{(x,t) \in \partial C}$,
so that $\Phi_{v(C)}(x)(\tilde{D}^{v(C_i)}_{(x,t)}) \subset  \tilde{D}^{v(C)}_{(x,t)}$
for all faces $C_i$ in $\partial C$.
We can then define an extension to a family $\{ \tilde{D}^{v(C)}_{(x,t)} \}_{(x,t) \in C}$
and set $D^{v(C)}_{(x,t)}=\Phi_{v(C)}^{-1}(x)(\tilde{D}^{v(C)}_{(x,t)})$.
As in the proof of Proposition \ref{hair} our construction guarantees that
the inductive properties are satisfied.
\end{proof}

Recall that in Proposition \ref{lsargument}
we obtained a relative cubical $k$-cycle $\alpha_i$,
such that for the minimizing sequence $\Phi_i \in \Pi$
and all $i$ large enough we have
$$[(\Phi_i)_*(\alpha_i)]=[(\Phi_i)_*(\alpha)]\in H_k(\mathcal{T},\mathcal{T}_{min}),$$
where $\alpha$ represents the $k$-chain $I^k \times \{0\}$.
Slightly abusing notation we will denote by $\alpha_i$
both the relative cycle and its support in $I^k \times [-1,1]$.
Let $ \Psi_i = \Phi_i|_{\alpha_i}$.
Applying Proposition \ref{hair_parametric} to the family $\{ \Psi_i(x)\}_{x \in \alpha_i}$ we obtain
\begin{thm}[Retracting to Hausdorff neighborhood of minimal tori]\label{hausdorff}
For each small $r>0$,  choosing $i$ large enough, there exists a continuous map
\begin{equation}
\Psi_i: \alpha_i\to \mathcal{T}
\end{equation}
so that
 $[(\Psi_i)_*(\alpha_i)]=H(a)\in H_k(\mathcal{T},\mathcal{T}_{min})$ with the following significance.   There exists a continuous function $g_i:\alpha_i\to \overline{\mathcal{T}}_{min}$ so that for any $b\in \alpha_i$ there holds
\begin{equation}
\Psi_i(b)\subset N_{r}(g_i(b)).
\end{equation}
Moreover, for $b\in\partial \alpha_i$ $g(b)$ is the torus given by $\Psi_i(b)$ (without its parametrization). 
\end{thm}

Since the results of this section are constructed via homotopies,  in the case $k=1$ we obtain from applying Proposition \ref{hair_parametric} to Proposition \ref{k=1case}  the stronger result
\begin{thm}[$k=1$ case]\label{hair_1}
For each small $r>0$,  there is a map  \begin{equation}a':[0,1]\to\mathcal{T}\end{equation} homotopic (rel $\partial [0,1]$) to the representative $a\in\pi_1(\mathcal{T},\mathcal{T}_{min})$ together with a continuous map $g:[0,1]\to\mathcal{\overline{T}}_{min}$ with the property that 
\begin{equation}
a'(t)\subset N_r(g(t))
\end{equation}
for all $t\in [0,1]$.  Furthermore, for $t\in\{0,1\}$, $g(t)$ is the torus given by $a(t)$.
\end{thm}

\section{Retracting tori in $T^2\times[-r,r]$}\label{section_finalretraction}
In this section, we complete the proof of Claim \ref{mainclaim}  and Claim \ref{mainclaim1} and thus Theorem \ref{main} by demonstrating: 
\begin{thm}\label{finalretract}
There holds:
\begin{enumerate}
\item  For large $i$,  the relative cycle $\Psi_i(\alpha_i)\in H_k(\mathcal{T},\mathcal{T}_{min})$ constructed in Theorem \ref{hausdorff} represents the trivial class in $H_k(\mathcal{T},\mathcal{T}_{min})$.   
\item For $k=1$, the map $a':[0,1]\to\mathcal{T}$ constructed in Theorem \ref{hair_1} is homotopic rel $\partial([0,1])$ to the trivial map in $\pi_1(\mathcal{T},\mathcal{T}_{min})$.  
\end{enumerate}
\end{thm}

The proof of Theorem \ref{finalretract} is based on techniques of Ivanov-Hatcher (\cite{I}, \cite{H2}) in the Haken case.   In particular,  it uses Hatcher's proof of the Smale conjecture for $\mathbb{S}^3$.

\subsection{Retracting to a single minimal torus}
Fix $q\in\mathcal{T}_{min}$.   Then by Lemma \ref{equal2},  $q: T^2\to M$ is given by $I_0\circ f_0$ for a unique isometry $I_0\in\mbox{isom}(M)$. (Recall that
$\{ f_t\}_{t \in (-1,1)}$ is the family of embeddings defined in Section \ref{section_main}.) Let $C_q\subset M$ denote the image of the embedding $q$ in $M$.   The torus $C_q$ inherits an orientation from the embedding $q$.    Henceforth we will assume that $C_q$ has been so oriented.  

For $r$ sufficiently small,  there is a diffeomorphism $D$ 
\begin{equation}D:T^2\times[-r,r]\to N_r(C_q),\end{equation}
with $D|_{T^2 \times  \{0\} } = q$.
For $-r\leq a\leq b\leq r$,  let $N_{a,b}(C_q)$ denote $D(C_q\times [a,b])$.  

As before, let $\mathcal{T}_{min}(q)=\mbox{isom}^+(M,q)$ denote the subspace of $\mathcal{T}_{min}$ consisting of isometries $I\in\mbox{isom}(M)$ so that \begin{enumerate}
\item the image of $I\circ f_0$ coincides with $C_q$; 
\item the induced orientations on $C_q$ coincide.    
\end{enumerate}
For each lens spaces $M\neq \mathbb{RP}^3$, $\mathcal{T}_{min}(q)$ is a torus.  For $M=\mathbb{RP}^3$ (recalling the discussion in Section \ref{spaceoftori}) $\mathcal{T}_{min}(q)$ consists of two tori. 

Let us denote by $\partial^-N_r(C_q)$ the boundary component of $N_r(C_q)$ which coincides with the torus 
$D(C_q \times \{-r\})$.

Let $\mbox{Emb}_0(T^2,N_r(C_q))$ denote the subset of embeddings $\mbox{Emb}_0(T^2,M)$ 
whose images are contained and homologically non-trivial in $N_r(C_q)$.   The goal of this section is to show that the inclusion \begin{equation}\iota:\mathcal{T}_{min}(q)\to \mbox{Emb}_0(T^2,N_r(C_q))\end{equation} is a homotopy equivalence on each path component.

Let $C_{free}$ denote the space of embeddings \begin{equation}e:T^2\times [-r,0]\to N_r(C_q)\end{equation} so that $e(T^2\times\{-r\})
= \partial N^-_r(C_q)$ and satisfies $e|_{T^2\times\{0\}} \in \mbox{Emb}_0(T^2,N_r(C_q))$.  The space $C_{free}$ denotes the space of ``partially free" collars - i.e.  collar neighborhoods where the lower side is free to slide.  There is a natural fibration
\begin{equation}
F \to C_{free}\to\mbox{Emb}_0(T^2,  N_r(C_q)).
\end{equation}
The map $p$ from $C_{free}$ is given by restriction $e(*,0):T^2\to N_r(C_q)$. From the theory of normal surfaces and Alexander's theorem \cite{hatcher2007notes}, every
torus in $\mbox{Emb}_0(T^2,  N_r(C_q))$ is isotopic to each of the boundary components of $N_r(C_q)$.
By the isotopy extension theorem, the region in $N_r(C_q)$ on the ``bottom" side of a torus in $\mbox{Emb}_0(T^2,  N_r(C_q))$ is diffeomorphic to $T^2\times [-r,0]$ and thus the map $p$ is surjective.

The fiber $p^{-1}(q)$ is given by $\mbox{Diff}_{free}(N_{-r,0}(C_q))$, which denotes the space of diffeomorphisms that fix $D(T^2\times\{0\})$ (i.e. ``the top face" of $N_{-r,0}(C_q)$) pointwise.  By Lemma 2 in \cite{H2},   $\mbox{Diff}_{free}(N_{-r,0}(C_q))$ is contractible.\footnote{This step uses the Smale conjecture for $\mathbb{S}^3$.}   Since the fiber of the fibration is contractible, it follows from the long exact sequence of the fibration and Whitehead's theorem that \begin{equation}\label{eq}p:C_{free}\to \mbox{Emb}_0(T^2,  N_r(C_q))\end{equation}
is a homotopy equivalence.   

We now consider a distinguished subspace of $C_{free}$.   Let $C_{min}(q)\subset C_{free}$ denote the collars satisfying
\begin{equation}
e(x,t) = I\circ I_0\circ f_t(x) \mbox{ for some } I\in\mathcal{T}_{min}(q).
\end{equation}
 Note that $p$ restricted to $C_{min}(q)$ maps homeomorphically onto $\mathcal{T}_{min}(q)$. 
Thus we have a commutative diagram (with $\iota_1$ and $\iota_2$ the inclusion maps):
\[
\begin{tikzcd}\label{commutative2}
C_{min}(q) \arrow[r, "\iota_1"] \arrow[d, "p_1"]
& C_{free} \arrow[d, "p" ] \\
\mathcal{T}_{min}(q) \arrow[r, "\iota_2"]
& |[rotate=0]| \mbox{Emb}_0(T^2,N_r(C_q))
\end{tikzcd}
\]

By Theorem \ref{basichomotopy} we obtain that $p$ induces isomorphisms on relative homotopy groups:
\begin{equation}\label{collars}
\pi_k(C_{free},C_{min}(q))\cong \pi_k(\mbox{Emb}_0(T^2,N_r(C_q)), \mathcal{T}_{min}(q)).  
\end{equation}
for all $k$.

Given a family of tori,  we obtain from \eqref{collars} a corresponding family of collars.   There is then a natural way to retract the tori onto a single standard one (up to rotation) by ``sliding along the collar."  

        We have the following: 

\begin{prop}[Retraction through collars]\label{retract}

Fix $k\geq 0$.  Suppose $f:D^{k}\to \mbox{Emb}_0(T^2,N_r(C_q))$ with $f|_{\partial D^{k}}$ contained in $\mathcal{T}_{min}(q)$.   Then $f$ is homotopic rel $\partial D^{k}$ to a map $g:D^{k}\to \mbox{Emb}_0(T^2,N_r(C_q))$ whose image consist of embeddings each contained in $\mathcal{T}_{min}(q)$.  

Thus \begin{equation}\label{embeddings}
 \pi_k(\mbox{Emb}_0(T^2,N_r(C_q)), \mathcal{T}_{min}(q))=0, 
\end{equation}
for all $k$,  which implies that the inclusion
\begin{equation}\label{ishomotopy}
\iota_2:\mathcal{T}_{min}(q)\to \mbox{Emb}_0(T^2,N_r(C_q))\end{equation} is a homotopy equivalence.  
\end{prop}

\begin{proof}
Let us first show \eqref{embeddings} for $k=0$.  In other words, we must show $\iota_2$ is surjective on path components. Given $E\in \mbox{Emb}_0(T^2,N_r(C_q))$, let $c\in C_{free}$, be such that $p(c) = E$. The collar $c$ gives a path of embeddings $\{c(*,s)\}_{s\in [-r,0]}$ (parameterized backwards) beginning at $E=c(*,0)$ and ending at an embedding $\overline{E}=c(*,-r)$ with the property that $\overline{E}:T^2\to\partial^-N_r(C_q)$.  We may then move $\overline{E}$ normally to an embedding $\tilde{E}:T^2\to C_q$. Given $I\in\mbox{isom}^+(M,q)$, we thus have $f_0^{-1}\circ I^{-1}\circ \tilde{E}\in\mbox{Diff}(T^2)$.

If $M\neq\mathbb{RP}^3$, then since the Goeritz group $G_0(M,C_q)$ is trivial (Proposition \ref{goeritz}), it follows that $f_0^{-1}\circ I^{-1}\circ \tilde{E}$ is contained in $\mbox{diff}(T^2)$ (i.e. the identity component of $\mbox{Diff}(T^2)$). Thus we may homotope $\tilde{E}$ to an element of $\mathcal{T}_{min}(q)$ through embeddings $T^2\to C_q$.  It follows that $\iota_2$ is surjective on path components in this case.

If $M=\mathbb{RP}^3$, choose $I_1$ and $I_2$ in the two connected components of $\mbox{isom}^+(M,q)$.  Then $f_0^{-1}\circ I_1^{-1}\circ \tilde{E}$ and $f_0^{-1}\circ I_2^{-1}\circ\tilde{E}$ lie in distinct components of $\mbox{Diff}(T^2)$.  Recall
\begin{equation}\mbox{Diff}(T^2)\simeq T^2\times\mbox{GL}(2,\mathbb{Z})\end{equation}
Since $G_0(\mathbb{RP}^3, C_q)=\mathbb{Z}_2$ (Proposition \ref{goeritz}), exactly one of these maps lies in the identity component $\mbox{diff}(T^2)\subset \mbox{Diff}(T^2)$
corresponding to $I\in\mbox{GL}(2,\mathbb{Z})$ and the other lies in the component corresponding to $-I\in\mbox{GL}(2,\mathbb{Z})$.  Thus we may further homotope $\tilde{E}$ to $\mathcal{T}_{min}(q)$ as above after choosing $i\in\{1,2\}$ so that $f_0^{-1}\circ I_i^{-1}\circ\tilde{E}\in\mbox{diff}(T^2)$.  This completes the proof that \eqref{embeddings} holds for $k=0$.

By \eqref{collars} it remains to show that $\pi_k(C_{free},C_{min}(q)) =0$
for all $k\geq 1$.   
Let $h:D^{k+1}\to C_{free}$ be a map with $h(\partial D^{k+1}) \subset C_{min}(q)$.  In the case of $k=1$, and $M=\mathbb{RP}^3$, since the Goeritz group of $T^2\times [-1,1]$ is trivial, both endpoints of $h(D^1)$ lie in the same component of $C_{min}(q)$. If $M\neq\mathbb{RP}^3$, the space $C_{min}(q)$ is connected.  

Let 
\begin{equation}
D(h(x)) = e_x: T^2 \times [-r,0] \rightarrow T^2 \times [-r,r],\end{equation}
where we may express in coordinates
\begin{equation}
e_x(y, \rho) =(e_x^1 (y, \rho), e_x^2 (y, \rho)).\end{equation}
First, we deform $\{(e_x^1 (y, \rho), e_x^2 (y, \rho)) \}$
to a family $\{(e_x^1 (y, \rho), \overline{e}_x^2 (y, \rho)) \}$,
satisfying $\frac{\partial \overline{e}_x^2}{\partial \rho} (y,-r) = 1 $.
Note that 
\begin{equation}\label{beg}
\overline{e}_x^2(y,-r) = -r\mbox{ for all } y\in T^2.
\end{equation}

We then define a homotopy (for $t\in [0,1]$)
$$\overline{e}_{(x,t)}(y,\rho) = \Big(e_x^1 (y, t(\rho+r)-r), \frac{\overline{e}_x^2 (y, t(\rho+ r) - r) +r}{t}-r \Big)$$
For $t =1$ we have $\overline{e}_{(x,1)}(y,\rho)= \overline{e}_x$ and by \eqref{beg} and L'Hospital's rule for $t \rightarrow 0$ we have
$$\overline{e}_{(x,0)}(y,\rho) = (e_x^1 (y,-r), \rho)$$
Define the homotopy \begin{equation}
h(x,t)=D^{-1}(\overline{e}_{(x,1-t)}).
\end{equation}
Note that this homotopy preserves the boundary condition $h(\partial D^{k+1},t) \subset C_{min}(q)$ for all $t$.
Finally, since each component of $\mbox{Diff}(T^2)$ retracts onto $C_{min}(q)$ (and thus the relative homotopy groups vanish), we can homotope the family
$\{ h(x,1)\}_{x \in D}$ to a family in $ C_{min}(q)$ (rel $\partial D$).
\end{proof}

\subsection{Retraction to a family of minimal tori}

Recall from Theorem \ref{hausdorff} that choosing $r$ small we have a map $\Psi_i: \alpha_i\to \mathcal{T}$ and $g_i:\alpha_i\to \mathcal{\overline{T}}_{min}$ so that
 \begin{equation}
\Psi_i(b)\subset N_r(g_i(b))
\end{equation}
for each $b\in \alpha_i$.   Furthermore,  for $b\in\partial \alpha_i$ we have
\begin{equation}
\Psi_i(b)\in\mathcal{T}_{min}.
\end{equation}




We have the main result of this section which we restate
\begin{thm}\label{finalretract1}
There holds:
\begin{enumerate}
\item  For large $i$,  the relative cycle $\Psi_i(\alpha_i)\in H_k(\mathcal{T},\mathcal{T}_{min})$ constructed in Theorem \ref{hausdorff} represents the trivial class in $H_k(\mathcal{T},\mathcal{T}_{min})$.   
\item For $k=1$, the map $a':[0,1]\to\mathcal{T}$ constructed in Theorem \ref{hair_1} is homotopic rel $\partial([0,1])$ to a trivial map in $\pi_1(\mathcal{T},\mathcal{T}_{min})$.  
\end{enumerate}
\end{thm}

The proof is an immediate corollary of the following.  

\begin{prop} \label{finalretractprop}
There exists a homotopy $F: [0,1] \times \alpha_i  \rightarrow \mathcal{T}$, such that 
\begin{enumerate}
\item $F(0,x) = \Psi_i(x)$ 
\item $F(1,x) \subset \mathcal{T}_{min}$.
\item For $x\in\partial\alpha_i$,  $F(t,x)=\Psi_i(x)$ for all $t\in [0,1]$.
\end{enumerate}
\end{prop}

The proof of Proposition \ref{finalretractprop} is similar to that of Proposition \ref{hair_parametric} (where we use Proposition \ref{retract} in place of Proposition \ref{hair}).
\begin{proof}
    Fix a cubulation of the relative cycle $\alpha_i$. By Proposition \ref{retract}
    for each $x$ in the $0-$skeleton of the cubulation
    we can define a contraction of $f(x)$ to $\mathcal{T}_{min}$.

    Assume we have defined the homotopy $F$
    on the $(k-1)-$skeleton of $\alpha_i$.
    Let $\Delta$ be $k$-cell and let $U = \Delta \cup (\partial \Delta \times [0,1])$. To extend $F$
    to $\Delta \times [0,1]$ we will define a homotopy
    of $F|_{U}$
    that fixes $F$ on $ \partial U =  \partial \Delta \times \{ 1 \}$.

Consider the fibration $ \Phi:\mbox{isom}(M) \rightarrow \overline{\mathcal{T}}_{min}$.
     Let $g: \alpha_i\to\overline{\mathcal{T}}_{min}$
    denote the function from Lemma \ref{g}. Since $U$ is contractible,
    there exists a lift of the map 
    $g|_U$ to a map $\tilde{g}: U \rightarrow\mbox{isom}(M)$. 
    
    Let $(\tilde{g}(x))^{-1}$ denote the inverse of the isometry $\tilde{g}(x)$. Consider the family of surfaces
    $\Sigma_x = (\tilde{g}(x))^{-1}(F(x))$, $x \in U$.
    By Proposition
    \ref{retract} there exists a contraction $F'$ of $\Sigma_x$
    to $\mathcal{T}_{min}$
    that is constant on $\partial U$. 
    Then $F(x,t) = \tilde{g}(x)(F'(x,t))$ defines the
    desired homotopy.
\end{proof}

Theorem \ref{finalretract1} completes the proofs of Claims \ref{mainclaim} and \ref{mainclaim1}.    This completes the proof of Theorem \ref{main}.  

\section{Appendix}
The following (relative) version of Whitehead's theorem for pairs (cf. Theorem 1.4.7 in \cite{Baues}) is standard but we include it for the reader's convenience. 

For $A\subset X$ and $B\subset Y$ we write \begin{equation}f:(X,A)\to (Y,B)\end{equation} and say $f$ \emph{is a map of pairs} if $f:X\to Y$ and $f(A)\subset B$.

\begin{thm}[Whitehead theorem for pairs]\label{basichomotopy}
Suppose $f:X\to Y$ is a homotopy equivalence of CW-complexes and $A\subset X$ and  $B\subset Y$ are subcomplexes with $f(A)=B$ and $f|_A:A\to B$ a homeomorphism. Then for all $k\geq 1$,  the induced map
\begin{equation}
f_*: \pi_k(X,A)\to \pi_k(Y,B). 
\end{equation}
is an isomorphism.  
Moreover, there exists a map of pairs \begin{equation}g: (Y,B)\to (X,A)\end{equation} so that $f\circ g: (X,A)\to (X,A)$ and  $g\circ f: (Y,B)\to (Y,B)$ are each homotopic to the identity map through maps of pairs $(X, A) \rightarrow (X,A)$ and $(Y, B) \to (Y,B)$ (respectively).

\end{thm}
\begin{proof}
There is a long exact sequence
\begin{equation}
...\to\pi_k(A)\to\pi_k(X)\to\pi_k(X,A)\to\pi_{k-1}(A)\to\pi_{k-1}(X)\to...
\end{equation}
as well as
\begin{equation}
...\to\pi_k(B)\to\pi_k(Y)\to\pi_k(Y,B)\to\pi_{k-1}(B)\to\pi_{k-1}(Y)\to..
\end{equation}
The map $f$ induces vertical downward arrows from the first long exact sequence giving rise to commutative diagrams (the maps commute in each square).   Since $f$ induces isomorphisms from $\pi_k(X)$ to $\pi_k(Y)$ and $\pi_k(A)$ to $\pi_k(B)$,  the five lemma implies that the middle maps
 \begin{equation}
f_*: \pi_k(X,A)\to \pi_k(Y,B). 
\end{equation}
are isomorphisms for each $k$.
\end{proof}

\printbibliography
\end{document}